\documentclass[12pt]{amsart}

\usepackage{amsfonts, amsthm, amsmath, amssymb}

\usepackage{rotating}
\usepackage{tikz}

\usepackage{amscd}

\usepackage[latin2]{inputenc}

\usepackage{t1enc}

\usepackage[mathscr]{eucal}

\usepackage{indentfirst}

\usepackage{graphicx}

\usepackage{graphics}

\usepackage{pict2e}

\usepackage{mathrsfs}

\usepackage{enumerate}
\usepackage{hyperref}
\hypersetup{backref, pagebackref, colorlinks=true}
\usepackage{cite}
\usepackage{color}
\usepackage{epic}
\usepackage{ferrers}
\usepackage{pdfsync} 
\numberwithin{equation}{section}
\theoremstyle{plain}

\newtheorem{theorem}{Theorem}[section]

\newtheorem{corollary}[theorem]{Corollary}

\newtheorem{proposition}[theorem]{Proposition}

\theoremstyle{definition}

\newtheorem{Def}[theorem]{Definition}

\newtheorem{example}[theorem]{Example}

\newtheorem{remark}[theorem]{Remark}

\newtheorem{?}[theorem]{Problem}

\newcommand{\N}{\mathbb{N}}

\newcommand{\R}{\mathbb{R}}

\def\R{\mathbb{R}}

\def\odd{\mathrm{odd}}

\newcommand\qbin[3]{\left[\begin{matrix} #1 \\ #2 \end{matrix} \right]_{#3}}

\def\boxit#1{\leavevmode\hbox{\vrule\vtop{\vbox{\kern.33333pt\hrule
    \kern1pt\hbox{\kern1pt\vbox{#1}\kern1pt}}\kern1pt\hrule}\vrule}}

\usepackage{collectbox}

\begin{document}
\title[unified approach]
{A unifying combinatorial approach to refined little G\"{o}llnitz and Capparelli's
companion  identities}
\author[S. Fu]{Shishuo Fu}
\address[Shishuo Fu]{College of Mathematics and Statistics, Chongqing University, Huxi campus building LD506, Chongqing 401331, P.R. China} 
\email{fsshuo@cqu.edu.cn}

\author[J. Zeng]{Jiang Zeng}

\address[Jiang Zeng]{Univ Lyon, Universit\'e Claude Bernard Lyon 1, CNRS UMR 5208, Institut Camille Jordan, 43 blvd. du 11 novembre 1918, F-69622 Villeurbanne cedex, France}
\email{zeng@math.univ-lyon1.fr}
  
\date{\today}

\begin{abstract}
Berkovich-Uncu have recently proved a
companion  of the well-known Capparelli's identities as well as 
refinements of Savage-Sills' new little G\"{o}llnitz identities. Noticing the connection between their results  and  
Boulet's earlier four-parameter partition generating functions, we discover
a new class of partitions, called $k$-strict partitions, to generalize their results.
By applying both horizontal and vertical dissections of Ferrers' diagrams with appropriate labellings, we provide  a unified combinatorial treatment 
of their results and shed more lights on the intriguing conditions of their companion to Capparelli's identities.

%

\end{abstract}

\maketitle

\tableofcontents

\section{Introduction}
 In 2006, following Andrews~\cite{and4} and  Stanley~\cite{st2}, Boulet \cite{bou} considered a four-variable generalization of Euler's  generating function  for the partition function by filling the cells of each odd-indexed (even-indexed) row of 
the diagram of a partition by $a$ and $b$ (resp. $c$ and $d$) cyclically 
and established the following:
\begin{align}\label{boulet}
\Phi(a,b,c,d) &:=\sum\limits_{\pi\in\mathcal{P}} \omega^2_{\pi}(a,b,c,d):=\frac{(-a,-abc;Q)_{\infty}}{(Q,ab,ac;Q)_{\infty}}, \quad Q:=abcd,
\end{align} 
where the weight $\omega^2_{\pi}(a,b,c,d):=a^{\# a}b^{\# b}c^{\# c}d^{\# d}$ is the product  of all the labels in the fillings of $\pi$'s diagram 
  as shown in the first diagram of Figure~\ref{2weights}, with 
$\# a$ denoting the number of cells labelled as $a$.  Throughout this paper we use $\mathcal{P}$ (resp. $\Phi$) to denote the set (resp. generating function) of ordinary partitions and adopt the standard $q$-notations~\cite{gara}:
\begin{align*}
(a;q)_0  :=1,\quad
(a;q)_k & := \prod_{i=1}^{k}(1-aq^{i-1}),\quad k\in \N^*\cup{\{\infty\}}\\
(a_1,a_2,\ldots,a_m;q)_s & := (a_1;q)_s(a_2;q)_s\ldots(a_m;q)_s.
\end{align*}

\begin{figure}
\begin{tikzpicture}[scale=0.5]
\draw (0,3) grid (10,5);
\draw (0,2) grid (7,3);
\draw (0,1) grid (5,2);
\draw (0,0) grid (2,1);
\draw (12,3) grid (22,5);
\draw (12,2) grid (19,3);
\draw (12,1) grid (17,2);
\draw (12,0) grid (14,1);
\draw (0.5,4.5) node{$a$};
\draw (1.5,4.5) node{$b$};
\draw (2.5,4.5) node{$a$};
\draw (3.5,4.5) node{$b$};
\draw (4.5,4.5) node{$a$};
\draw (5.5,4.5) node{$b$};
\draw (6.5,4.5) node{$a$};
\draw (7.5,4.5) node{$b$};
\draw (8.5,4.5) node{$a$};
\draw (9.5,4.5) node{$b$};
\draw (0.5,2.5) node{$a$};
\draw (1.5,2.5) node{$b$};
\draw (2.5,2.5) node{$a$};
\draw (3.5,2.5) node{$b$};
\draw (4.5,2.5) node{$a$};
\draw (5.5,2.5) node{$b$};
\draw (6.5,2.5) node{$a$};
\draw (0.5,0.5) node{$a$};
\draw (1.5,0.5) node{$b$};
\draw (0.5,3.5) node{$c$};
\draw (1.5,3.5) node{$d$};
\draw (2.5,3.5) node{$c$};
\draw (3.5,3.5) node{$d$};
\draw (4.5,3.5) node{$c$};
\draw (5.5,3.5) node{$d$};
\draw (6.5,3.5) node{$c$};
\draw (7.5,3.5) node{$d$};
\draw (8.5,3.5) node{$c$};
\draw (9.5,3.5) node{$d$};
\draw (0.5,1.5) node{$c$};
\draw (1.5,1.5) node{$d$};
\draw (2.5,1.5) node{$c$};
\draw (3.5,1.5) node{$d$};
\draw (4.5,1.5) node{$c$};

\draw (12.5,4.5) node{$a$};
\draw (13.5,4.5) node{$b$};
\draw (14.5,4.5) node{$c$};
\draw (15.5,4.5) node{$a$};
\draw (16.5,4.5) node{$b$};
\draw (17.5,4.5) node{$c$};
\draw (18.5,4.5) node{$a$};
\draw (19.5,4.5) node{$b$};
\draw (20.5,4.5) node{$c$};
\draw (21.5,4.5) node{$a$};
\draw (12.5,2.5) node{$a$};
\draw (13.5,2.5) node{$b$};
\draw (14.5,2.5) node{$c$};
\draw (15.5,2.5) node{$a$};
\draw (16.5,2.5) node{$b$};
\draw (17.5,2.5) node{$c$};
\draw (18.5,2.5) node{$a$};
\draw (12.5,0.5) node{$a$};
\draw (13.5,0.5) node{$b$};
\draw (12.5,3.5) node{$d$};
\draw (13.5,3.5) node{$e$};
\draw (14.5,3.5) node{$f$};
\draw (15.5,3.5) node{$d$};
\draw (16.5,3.5) node{$e$};
\draw (17.5,3.5) node{$f$};
\draw (18.5,3.5) node{$d$};
\draw (19.5,3.5) node{$e$};
\draw (20.5,3.5) node{$f$};
\draw (21.5,3.5) node{$d$};
\draw (12.5,1.5) node{$d$};
\draw (13.5,1.5) node{$e$};
\draw (14.5,1.5) node{$f$};
\draw (15.5,1.5) node{$d$};
\draw (16.5,1.5) node{$e$};
\draw (5,-1) node{$\omega^2_{\pi}(a,b,c,d)=a^{10}b^{9}c^{8}d^{7}$};
\draw (17,-1) node{$\omega^3_{\pi}(a,b,c,d,e,f)=a^{8}b^{6}c^{5}d^{6}e^{5}f^{4}$};

\end{tikzpicture}
\caption{Partition $\pi=(10,10,7,5,2)$ with weights $\omega^2_{\pi}$ and $\omega^3_{\pi}$}
\label{2weights}
\end{figure}

 Boulet~\cite{bou}  also obtained the strict version of \eqref{boulet}:
\begin{align}\label{boulet-strict}
\Psi(a,b,c,d) &:=\sum_{\pi\in \mathcal{D}} \omega^2_{\pi}(a,b,c,d)=\frac{(-a,-abc;Q)_{\infty}}{(ab;Q)_{\infty}}, \quad Q:=abcd,
\end{align}
where we use $\mathcal{D}$ (resp. $\Psi$)
 to denote the set (resp.  generating function) of strict partitions.

If $\pi=(\pi_1, \pi_2, \ldots)$ is a partition, we denote by $|\pi|$ the sum of its parts and 
by ${\odd}(\pi)$  the number of its odd  parts,
$\pi_o$ (resp. $\pi_o$) the partition consisting of the odd-indexed (resp. even-indexed) parts of $\pi$.  Now, with the substitution $(a,b,c,d)=(xt, x/t, yz, y/z)$ in (\ref{boulet-strict}) we have 
\begin{align}\label{boulet-strict-ss-bu}
\sum_{\pi\in \mathcal{D}} x^{|\pi_\textrm{o}|}y^{|\pi_\textrm{e}|}t^{{\odd}(\pi_\textrm{o})}z^{{\odd}(\pi_\textrm{e})}=\frac{(-xt,-x^2yz;x^2y^2)_{\infty}}{(x^2;x^2y^2)_{\infty}}.
\end{align}
We would like to point out  that  the above identity encompasses 
several  results in the recent literature as special cases.
For example,  the two special ($z=0$ or $t=0$) 
cases of  \eqref{boulet-strict-ss-bu} 
correspond to 
 Theorem~4.3 and Theorem~4.4 of \cite{ss}, respectively, which
  imply in particular their  new little G\"ollnitz identities.
  \begin{theorem}[Savage-Sills]
 The number of partitions of $n$ into distinct parts in which even-indexed $($resp. odd-indexed$)$ parts are even is equal to the number of partitions of $n$ into parts  $\equiv 1,5,6\pmod{8}$ $($resp. $2,3,7\pmod{8}$$)$.
\end{theorem}
 
 In view of 
Euler's formula $(-q;q)_\infty=1/(q;q^2)_\infty$, 
the  $x=y=q$ case of  \eqref{boulet-strict-ss-bu}  reduces to 
\begin{align}\label{boulet-stric-bu}
\sum_{\pi\in \mathcal{D}} q^{|\pi|}t^{{\odd}(\pi_\textrm{o})}z^{{\odd}(\pi_\textrm{e})}=(-qt,-q^3z, -q^2, -q^4;q^4)_{\infty}.
\end{align}
This is  equivalent to Berkovich and Uncu's Theorem 1.1 in \cite{bu2}, while 
the special $z=1$ ($t=1$) case of 
\eqref{boulet-stric-bu} corresponds to  Theorem 2.4 of \cite{bu2}. 
\begin{theorem}[Berkovich and Uncu]\label{bu-2}
The number of partitions of $n$ into distinct parts with $i$ odd-indexed odd parts and $j$
even-indexed odd parts is equal to 
 the number of partitions of $n$ into distinct parts with $i$ parts
that are congruent to 1 modulo 4, and $j$ parts that are congruent to 3 modulo 4.
\end{theorem}

Actually they proved a finite version~\cite[Theorem~4.1]{bu2}  of the above result 
using recurrence and a special case of a finite version of 
\eqref{boulet-strict} due to Ishikawa and the second author~\cite[Corollary 3.4]{iz}.
  One of our  aims is to give a combinatorial proof of their finite version
 using a variant of Boulet's bijection (see Section~\ref{BC1}).

Another impetus of this work is the connection of \eqref{boulet-strict-ss-bu} with Capparelli's identities \cite{cappa}.
In 1988, Capparelli conjectured in his thesis \cite{cappath} two Rogers-Ramanujan type identities, which are described by Alladi et al. \cite{aag} as ``new and quite subtle''. Andrews proved the first identity in 1994 \cite{and3} via generating function manipulation, with Lie-theoretic proofs supplied later by Tamba and Xie \cite{tx} and by Capparelli himself \cite{cappa}. Finally in 1995, both identities were proven by Alladi, Andrews and Gordon \cite{aag}. For recent study on Capparelli's identities, see for example \cite{sil1,dou,bm,dl}.
In Capparelli's original identities there are the 
 infinite product sides or the modular sides and the ``gap condition'' sides.
 Berkovich and Uncu \cite{bu1}  defined  a new ``gap condition''. 
For completeness we first quote their definition and one of their results below.
\begin{Def}
\label{level3}
For $m\in\{1,2\}$, let $A_m(n)$ be the number of partitions $\pi=(\pi_1,\pi_2,\ldots)$ of $n$ such that
\begin{enumerate}[i.]
\item $\pi_{2j+r}\not\equiv 3-m+(-1)^mr \pmod 3$, 
\item $\pi_{2j+r}-\pi_{2j+r+1}>\lfloor m/2\rfloor+(-1)^{m-1}r$ for $1\leq 2j+r$,
\end{enumerate}
where $r\in \{0,1\}$ and $j\in \N$; and
let $C_m(n)$ be the number of partitions of $n$ into distinct parts $\not\equiv \pm m\pmod{6}$.
\end{Def}
As remarked by Berkovich and Uncu \cite{bu1} the second condition of $A_m(n)$  can be replaced with the condition that 
all parts are distinct and $3l+1$ and $3l+2$ do not appear together as consecutive parts for any integer $l\geq 0$.
\begin{theorem}[Berkovich-Uncu]\label{bu-A=C}
For $m\in \{1,2\}$ and positive integers $n$, we have 
\begin{align}\label{bunewcomp}
A_m(n)=C_m(n).
\end{align}
\end{theorem}
To prove the result,  as for Theorem~\ref{bu-2}, they derived a finite version of the above identity using recurrence relations and 
proved a finite analogue~\cite[Theorem 2.5]{bu1} of \eqref{bunewcomp}.
At the end of their paper, Berkovich-Uncu noticed  that  one can obtain \eqref{bunewcomp} as a non-trivial 
corollary of Boulet's results~\cite{bou} and made the suggestion on extending Boulet's 
work \cite{bou} to deal with its finite version. In answering their request, we give a similar construction in the case of modulo $3$ and $6$, which runs parallel to Boulet's case of modulo $2$ and $4$. To deal with our case, we need to introduce a different weight with six parameters $\omega^3_{\pi}(a,b,c,d,e,f)$, which we include in Figure~\ref{2weights} as well for easy comparison.



The main goal of this paper is to combinatorially establish the weighted generating functions for a special class of partitions that we call ``$k$-strict'' and then demonstrate the unifying nature of this approach in the case of $k=2$ and $k=3$.

In Section~\ref{sec: kstrict}, we define $k$-strict partitions and introduce a key decomposition of partitions, then we combinatorially deduce the aforementioned generating functions (see Theorem~\ref{main1}) and their specializations (see Theorem~\ref{main2}). A further specialization leads to a generalization of Berkovich and Uncu's new companion of Capparelli's identities (see Theorem~\ref{bu-k}). Next in Section~\ref{sec: mod3}, we see the first application of our construction, which produces one identity (see \eqref{gfBevaid}) that includes Berkovich-Uncu's new companion of Capparelli's identities as two special cases. Section~\ref{sec: mod2} presents another application, which results in a combinatorial proof of  a previous result (see (\ref{iz1}, \ref{iz2}))
of Ishikawa and the second author~\cite{iz},
and we explain the connections between our methods and the existing proofs, then continue to discuss the more general doubly-bounded case. Finally, in Section~\ref{sec: modk} we conclude with some remarks.   


\section{$k$-strict Partitions and main results}\label{sec: kstrict}
In this section, we introduce a new class of partitions as well as a key decomposition that will be the main tools to obtain all of our results. This novel class of partitions is in some sense broader than Euler's strict partition. To make it precise, we give the following definition.
\begin{Def}\label{kstrictdef}
Given an integer $k\geq 1$, we call a partition $\pi$ ``$k$-strict'' if 
at most one part occurs in each block $\{mk+1, \ldots, mk+k-1\}$ with $m\in \N$, in other words,
if for any integers $r_1,r_2$, with $ 1\leq r_1\leq r_2\leq k-1$,
\begin{align}\label{kstrictcond}
mk+r_1\quad \text{and} \quad mk+r_2 \quad\text{do not appear together as   parts in } \pi.
\end{align} 
\end{Def}

The ``$1$-strict'' partitions are just ordinary partitions because $1\leq r_1\leq r_2\leq k-1=0$ voids  condition (\ref{kstrictcond}), while  ``$2$-strict'' partitions are those partitions with odd parts all distinct.
Note that the later partitions have been thoroughly studied in the literature; see for example,  Alladi \cite{alla}, Andrews \cite{and1,and2} and Hirschhorn-Sellers \cite{hs}. For  $k\geq 3$ the notion of "k-strict" partitions seems new.
For example, there are nine $3$-strict partitions of $10$:
\begin{gather*}
(10), (9,1), (8,2), (7,3), (6,4),
(6,3,1), (5,3,2), (4,3,3), (3,3,3,1).
\end{gather*}

\begin{Def}
Let $\mathcal{S}^k$ be  the set of $k$-strict partitions and 
$\mathcal{E}^k$  the set of partitions with parts $\equiv 0 \pmod k$, and each part occurs even number of times.
\end{Def}

Clearly we have $\mathcal{E}^k\subset \mathcal{S}^k$.
By Definition~\ref{kstrictdef}, one checks easily that $\ \mathcal{D}\cap \mathcal{S}^1= \mathcal{D}\cap \mathcal{S}^2=\mathcal{D}$, 
but $ \mathcal{D}\cap \mathcal{S}^k\neq \mathcal{D}$ for $k\geq 3$. 
This  observation explains the simpler structure of $2$-strict case related with the new little G\"{o}llnitz identities (Section~\ref{sec: mod2}) and suggests  more intricate conditions for $k$-strict case with  $k\geq 3$ (see Section~\ref{sec: mod3} for $k=3$).  We denote $ \mathcal{D}\cap \mathcal{S}^k$ as $\mathcal{DS}^k$ for short. 

\begin{Def}[$\omega^k$-weight] Let $k\geq 1$ be a positive integer.
Given a partition $\pi$, 
we label 
the cells in the odd-indexed (resp. even-indexed) rows of $\pi's$ 
diagram cyclically from left to right with $a_1, a_2, \ldots, a_k$ (resp. $b_1, b_2, \ldots, b_k$) and define  the product of all the labels on the  diagram as its 
$\omega^k$-weight, denoted by
$\omega_{\pi}^{k}\big((a_i), (b_i)\big)$, see  Figure~\ref{2weights}
for two examples when $k=2$ and $k=3$.
\end{Def}

When no confusion is caused, we simply write $\omega^k$ by
suppressing  the labels $(a_i), (b_i)$. Now we are ready to describe our key decomposition $\psi_k$. Given a partition $\pi\in\mathcal{S}^k$, we repeatedly remove even copies of its repeated parts if any, which are necessarily $\equiv 0 \pmod k$  due to condition (\ref{kstrictcond}), then we are left with a partition, say $\pi^1\in \mathcal{DS}^k$, and all the removed parts form a new partition $\pi^2\in\mathcal{E}^k$. In order  to keep track of the weight $\omega^{k}$ associated with each partition, we state the following theorem.

\begin{theorem}\label{decomp1}
For any $k\geq 1$, the map $\psi_k: \pi \longmapsto (\pi^1, \pi^2)$ as described above is a weight-preserving bijection from $\mathcal{S}^k$ to 
$\mathcal{DS}^k\times \mathcal{E}^k$
such that $\ell(\pi)=\ell(\pi^1)+\ell(\pi^2)$ and
\begin{align}\label{decomp1-eq}
\omega^k_{\pi}\big((a_i), (b_i)\big) =\omega^k_{\pi^1}\big((a_i), (b_i)\big)\, \omega^k_{\pi^2}\big((a_i), (b_i)\big) , 
\end{align} 
where $\ell(\pi)$ stands for the number of parts of $\pi$.
\end{theorem}
\begin{proof}
Suppose we are given a $k$-strict partition $\pi$ with $\omega^k$-label, we take the following steps to obtain $\pi^1$ and $\pi^2$, also with $\omega^k$-label. We recommend Figure~\ref{pitopied} for illustration with one example of such decomposition when $k=3$.
\begin{itemize}
\item Step 1. If there are repeated parts, which are necessarily $\equiv 0 \pmod k$ remained in $\pi$, find the largest such part, say $\pi_t$, and suppose $\pi_t$ is repeated $m$ ($m\geq 2$) times. Otherwise jump to Step 3.
\item Step 2. Remove the first $2\lfloor {m}/{2} \rfloor$ appearances of $\pi_t$ from $\pi$. As for the labelling on these removed parts, if the first copy of $\pi_t$ is odd-indexed in $\pi$, then keep their original $\omega^k$-labels, if the first copy is even-indexed, then swap
$$
a_1\leftrightarrow b_1, a_2\leftrightarrow b_2, \ldots, a_k\leftrightarrow b_k \qquad (\star)
$$
for the labellings in these $2\lfloor {m}/{2} \rfloor$ copies, since these are even number of copies, so the total weight is preserved. Go back to Step 1.
\item Step 3. Collect all the parts removed in Step 2, together with their new labels to form partition $\pi^2$. Group the remaining parts in $\pi$ together with their original $\omega^k$-labels, call this new partition $\pi^1$.
\end{itemize}
Note that $\pi^2$ has $\omega^k$-label as a result of the modification ($\star$) we made in Step 2. And since in Step 2 we always remove even number of parts, the labelling on $\pi^1$ remains $\omega^k$-label as well. 
These two observations lead to $\omega^k_{\pi}=\omega^k_{\pi^1}\omega^k_{\pi^2}$
as well as to  $\ell(\pi)=\ell(\pi^1)+\ell(\pi^2)$. Every step of the construction is easily seen to be bijective. 
\end{proof}

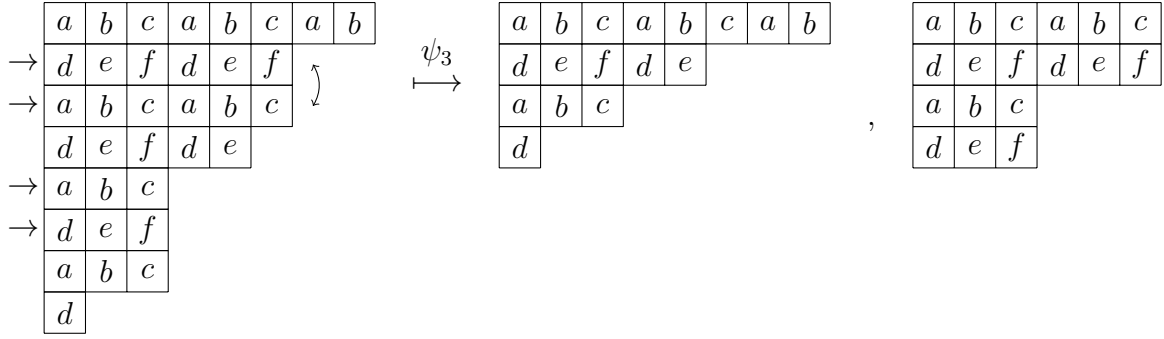
\begin{figure}
\begin{tikzpicture}[scale=0.55]
\draw (0,0) grid (1,1);
\draw (0,1) grid (3,2);
\draw (0,2) grid (3,3);
\draw (0,3) grid (3,4);
\draw (0,4) grid (5,5);
\draw (0,5) grid (6,6);
\draw (0,6) grid (6,7);
\draw (0,7) grid (8,8);

\draw (21,4) grid (24,5);
\draw (21,5) grid (24,6);
\draw (21,6) grid (27,7);
\draw (21,7) grid (27,8);
\draw (11,4) grid (12,5);
\draw (11,5) grid (14,6);
\draw (11,6) grid (16,7);
\draw (11,7) grid (19,8);

\draw (0.5,0.5) node{$d$};
\draw (0.5,1.5) node{$a$};
\draw (1.5,1.5) node{$b$};
\draw (2.5,1.5) node{$c$};
\draw (0.5,2.5) node{$d$};
\draw (1.5,2.5) node{$e$};
\draw (2.5,2.5) node{$f$};
\draw (0.5,3.5) node{$a$};
\draw (1.5,3.5) node{$b$};
\draw (2.5,3.5) node{$c$};
\draw (0.5,4.5) node{$d$};
\draw (1.5,4.5) node{$e$};
\draw (2.5,4.5) node{$f$};
\draw (3.5,4.5) node{$d$};
\draw (4.5,4.5) node{$e$};
\draw (0.5,5.5) node{$a$};
\draw (1.5,5.5) node{$b$};
\draw (2.5,5.5) node{$c$};
\draw (3.5,5.5) node{$a$};
\draw (4.5,5.5) node{$b$};
\draw (5.5,5.5) node{$c$};
\draw (0.5,6.5) node{$d$};
\draw (1.5,6.5) node{$e$};
\draw (2.5,6.5) node{$f$};
\draw (3.5,6.5) node{$d$};
\draw (4.5,6.5) node{$e$};
\draw (5.5,6.5) node{$f$};
\draw (0.5,7.5) node{$a$};
\draw (1.5,7.5) node{$b$};
\draw (2.5,7.5) node{$c$};
\draw (3.5,7.5) node{$a$};
\draw (4.5,7.5) node{$b$};
\draw (5.5,7.5) node{$c$};
\draw (6.5,7.5) node{$a$};
\draw (7.5,7.5) node{$b$};

\draw (21.5,4.5) node{$d$};
\draw (22.5,4.5) node{$e$};
\draw (23.5,4.5) node{$f$};
\draw (21.5,5.5) node{$a$};
\draw (22.5,5.5) node{$b$};
\draw (23.5,5.5) node{$c$};
\draw (21.5,6.5) node{$d$};
\draw (22.5,6.5) node{$e$};
\draw (23.5,6.5) node{$f$};
\draw (24.5,6.5) node{$d$};
\draw (25.5,6.5) node{$e$};
\draw (26.5,6.5) node{$f$};
\draw (21.5,7.5) node{$a$};
\draw (22.5,7.5) node{$b$};
\draw (23.5,7.5) node{$c$};
\draw (24.5,7.5) node{$a$};
\draw (25.5,7.5) node{$b$};
\draw (26.5,7.5) node{$c$};

\draw (11.5,4.5) node{$d$};
\draw (11.5,5.5) node{$a$};
\draw (12.5,5.5) node{$b$};
\draw (13.5,5.5) node{$c$};
\draw (11.5,6.5) node{$d$};
\draw (12.5,6.5) node{$e$};
\draw (13.5,6.5) node{$f$};
\draw (14.5,6.5) node{$d$};
\draw (15.5,6.5) node{$e$};
\draw (11.5,7.5) node{$a$};
\draw (12.5,7.5) node{$b$};
\draw (13.5,7.5) node{$c$};
\draw (14.5,7.5) node{$a$};
\draw (15.5,7.5) node{$b$};
\draw (16.5,7.5) node{$c$};
\draw (17.5,7.5) node{$a$};
\draw (18.5,7.5) node{$b$};

\draw (9.5,6) node{$\longmapsto$};
\draw (9.5,6.8) node{$\psi_3$};
\draw (20,5) node{,};
\path[<->] (6.5,5.5) edge [bend right=30] (6.5,6.5);
\draw (-0.5,2.5) node{$\rightarrow$};
\draw (-0.5,3.5) node{$\rightarrow$};
\draw (-0.5,5.5) node{$\rightarrow$};
\draw (-0.5,6.5) node{$\rightarrow$};

\end{tikzpicture}
\caption{Decomposition of $\pi=(8,6,6,5,3,3,3,1)$ into $(\pi^1,\pi^2)$ with $\omega^3$-labels}
\label{pitopied}
\end{figure}
We are now ready to compute the $\omega^k$-weight generating functions of the three sets $ \mathcal{E}^k$,
$\mathcal{S}^k$ and $\mathcal{DS}^k$ of partitions.
\begin{theorem}\label{main1}
For any integer $k\geq 1$, let $\{a_1,a_2,\ldots,a_k,b_1,b_2,\ldots,b_k\}$ be $2k$ commutable variables, and let
\begin{align*}
z_k &=a_1\ldots a_k, \quad w_k=a_1b_1\ldots a_kb_k,\\
x_k &=a_1+a_1a_2+\cdots+a_1\ldots a_{k-1},\\
y_k &=z_k(b_1+b_1b_2+\cdots+b_1\ldots b_{k-1}). 
\end{align*}
Then we have
\begin{align}
\sum\limits_{\pi\in \mathcal{E}^k}\omega_{\pi}^{k}\big((a_i), (b_i)\big) &=\dfrac{1}{(w_k;w_k)_{\infty}},\label{gfEk}\\
\sum\limits_{\pi\in\mathcal{S}^k}\omega_{\pi}^{k}\big((a_i), (b_i)\big) 
&=\dfrac{(-x_k,-y_k;w_k)_{\infty}}{(z_k,w_k;w_k)_{\infty}},\label{gfSk}\\
\sum\limits_{\pi\in\mathcal{DS}^k}\omega_{\pi}^{k}\big((a_i), (b_i)\big) &=\dfrac{(-x_k,-y_k;w_k)_{\infty}}{(z_k;w_k)_{\infty}}.\label{gfDSk}
\end{align}
\end{theorem}
\begin{proof}
Let  $\pi$ be  a  partition in  $\mathcal{E}^k$, then 
 each part of $\pi$ is a multiple of $k$ and repeated even times. 
 If $\pi=(\pi_1, \ldots, \pi_{2l})$,  we define $\pi^*=(\pi_1/k, \pi_3/k,\ldots, \pi_{2l-1}/k)$.
 Clearly the mapping $\pi\mapsto \pi^*$ is a bijection from $\mathcal{E}^k$ to $\mathcal{P}$ such that $\omega^k_\pi=(w_k)^{|\pi^*|}$. Identity~\eqref{gfEk} follows then from  the generating function of partitions.

\begin{figure}
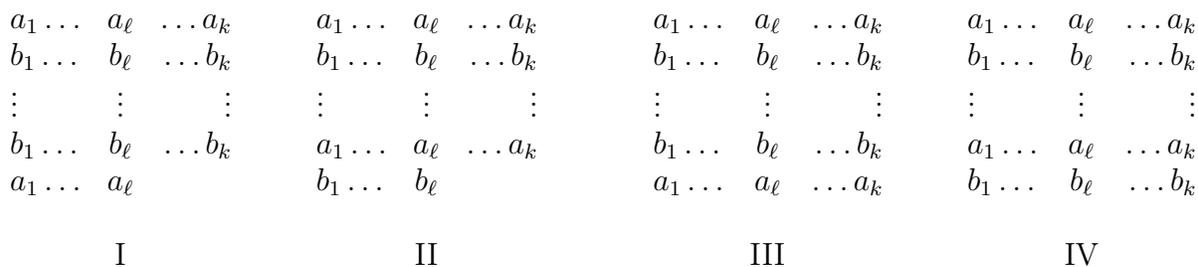

$$
\begin{array}{lcrc lcrc lcrc lcr}
a_1\ldots &a_\ell&  \ldots a_k & \quad & a_1\ldots &a_\ell & \ldots a_k& \quad &
a_1\ldots& a_\ell& \ldots a_k&\quad& a_1\ldots &a_\ell&  \ldots a_k \\
b_1\ldots & b_\ell&   \ldots b_k & \quad & b_1\ldots& b_\ell  & \ldots b_k& \quad &
b_1\ldots & b_\ell&   \ldots b_k&\quad &b_1\ldots&  b_\ell&   \ldots b_k\\
\vdots &  \vdots&   \vdots & \quad & \vdots&  \vdots&  \vdots & \quad &\vdots&  \vdots& \vdots &\quad&\vdots &  \vdots&   \vdots \\
b_1\ldots &b_\ell&   \ldots b_k & \quad &a_1\ldots& a_\ell&   \ldots a_k & \quad & 
b_1\ldots& b_\ell&   \ldots b_k&\quad & a_1\ldots& a_\ell&   \ldots a_k \\
a_1 \ldots&  a_\ell&   &   \quad & b_1  \ldots& b_\ell &   & \qquad & 
a_1\ldots& a_\ell&  \ldots a_k &\quad& b_1\ldots&  b_\ell&   \ldots b_k  \\
 &&&&&&\\
  &\text{I} &    & \quad &  &\text{II} &   & \quad &   &\text{III}&  &\quad &&\text{IV} &
\end{array}
$$
\caption{Four possible types of vertical blocks where $1\leq \ell\leq k-1$.}
\label{sixtypes}
\end{figure}

 Given any $\pi\in\mathcal{S}^k$, we decompose its labelled  diagram vertically into blocks of width $k$. 
 Due to condition (\ref{kstrictcond}) in Definition~\ref{kstrictdef}, it is not difficult to see  that there are only
the 4 types of blocks  as shown in Figure~\ref{sixtypes}. Moreover, the weight of the blocks of type
 I, II,  III and IV are respectively 
$$
(w_k)^{m} x_k,\quad  (w_k)^{m} y_k,
\quad z_k(w_k)^{m}, \quad (w_k)^{m+1}.
$$ 
where $2m+1$ (resp. $2m+2$) is the height of  the block of type I and III (resp. II and IV) in Figure~\ref{sixtypes}.

Now for each type of blocks, it is routine to give their generating functions, just note that for types 
I and II,  the blocks of the same type must have distinct length, while for types III and IV
 repetition is possible. Therefore, $(-x_k;w_k)_{\infty}$ generates type 
 I and $( -y_k;w_k)_{\infty}$   generates type II blocks, 
  $1/(z_k;w_k)_{\infty}$ generates type III blocks, 
 and $1/(w_k;w_k)_{\infty}$ generates type IV blocks. 
 Finally the generating function for all such $\pi\in\mathcal{S}^k$ 
   is the product of all 4 types, this establishes \eqref{gfSk}.
   Finally, in view of \eqref{decomp1-eq}  we derive  identity~\eqref{gfDSk} from \eqref{gfEk}  and \eqref{gfSk}.
\end{proof}

 For any given partition $\pi$ and $1\leq l\leq k-1$, we use $o_{l}(\pi)$ (resp. $e_{l}(\pi)$) to denote the number of odd-indexed (resp. even-indexed) parts that are $\equiv l \pmod k$. And recall $|\pi_{o}|$ and $|\pi_{e}|$ from \eqref{boulet-strict-ss-bu}.
\begin{theorem}\label{main2}
For any integer $k\geq 1$, we have
\begin{align}\label{gfDSkeva}
\sum\limits_{\pi\in\mathcal{DS}^k}x^{|\pi_{o}|}y^{|\pi_{e}|}\prod\limits_{l=1}^{k-1}u_l^{o_l(\pi)}v_l^{e_l(\pi)}
&=\dfrac{\biggl(-\sum\limits_{l=1}^{k-1}u_lx^l,-x^k\sum\limits_{l=1}^{k-1}v_ly^l;x^ky^k\biggr)_{\infty}}{(x^k;x^ky^k)_{\infty}}.
\end{align}
\end{theorem}
\begin{proof}
In \eqref{gfDSk}, simply take $a_l=u_lx/u_{l-1},b_l=v_ly/v_{l-1}$, for $l=1,\ldots,k$, where $u_0=u_k=v_0=v_k=1$.
\end{proof}

To produce some ``nice'' partition theorems, one needs to make further restrictions on the residue class modulo $k$ for odd-indexed parts and even-indexed parts separately. Essentially, one wants to reduce both sums $\sum\limits_{l=1}^{k-1}u_lx^l$ and $\sum\limits_{l=1}^{k-1}v_ly^l$ each to a single term. There are $(k-1)^2$ different ways this can be done. 
For general $k$, we only give one pair below to show the idea. 

\begin{Def}\label{kAC}
For integers $n\geq 0, k\geq 3$ and $m\in\{1,2\}$. We use $A^k_m(n)$ to denote the number of $k$-strict partitions of $n$ such that:
\begin{enumerate}[i.]
\item No parts can be repeated;
\item all the odd-indexed parts are $\equiv k\; \text{or}\;  3-m \pmod k$;
\item all the even-indexed parts are $\equiv k\; \text{or}\; m \pmod k$.
\end{enumerate}
And we use $C^k_m(n)$ to denote the number of partitions of $n$ into distinct parts which are
congruent  to
\begin{align}\label{Ckmn}
 3-m, k, k+m \text{ or } 2k\pmod {2k}.
\end{align}
\end{Def}
\begin{theorem}\label{bu-k}
For integers $n,i,j\geq 0, k\geq 3$ and $m\in\{1,2\}$, the number of partitions enumerated by $A^k_m(n)$ that have exactly $i$ parts $\equiv 3-m \pmod k$ and $j$ parts $\equiv m \pmod k$ 
equals the number of partitions enumerated by $C^k_m(n)$ that have exactly $i$ parts $\equiv 3-m \pmod {2k}$ and $j$ parts $\equiv k+m \pmod {2k}$. In particular, we have
\begin{align}\label{ACK}
A^k_m(n)=C^k_m(n). 
\end{align}
\end{theorem}
\begin{proof}
Let $\mathcal{DS}_m^k$ be the subset of partitions $\pi$ in $\mathcal{DS}^k$ satisfying 
 $o_i(\pi)=e_j(\pi)=0$ if $(i,j)\neq (3-m, m)$. 
 It is clear that the partitions in  $\mathcal{DS}_m^k$ are exactly the $k$-strict partitions 
 satisfying conditions i, ii and iii of Definition~\ref{kAC}.
Taking  $u_l=0$ for $l\neq 3-m$, $v_l=0$ for $l\neq m$ and $x=y=q$ in \eqref{gfDSkeva}  we obtain
\begin{align*}
\sum\limits_{\pi\in \mathcal{DS}_m^k}q^{|\pi|}u_{3-m}^{\textrm{o}_{3-m}(\pi)}v_m^{\textrm{e}_m(\pi)}
&=\dfrac{\bigl(-u_{3-m}q^{3-m},-v_mq^{k+m};q^{2k}\bigr)_{\infty}}{(q^k;q^{2k})_{\infty}}\\
&=
\bigl(-u_{3-m}q^{3-m},-q^k, -v_mq^{k+m}, -q^{2k};q^{2k}\bigr)_{\infty},
\end{align*}
then compare the coefficients of $u_{3-m}^iv_m^jq^n$ on both sides to get the first claim. Then summing over all $i,j$ we arrive at \eqref{ACK}.
\end{proof}

For example, when $n=12$ and $k=5$, the corresponding sets counted by the two sides 
of \eqref{ACK} are:
$$
\{(12), (7, 5)\}\quad \textrm{and} \quad \{(12), (10, 2)\}\quad \textrm{for $m=1$};
$$
and
$$
\{(10, 2), (6, 5, 1)\}\quad \textrm{and} \quad \{(11, 1), (7, 5)\}\quad \textrm{for $m=2$}.
$$
\begin{remark} 

As an afterthought, we can also prove \eqref{ACK} directly with the choice $(a,b,c,d) = (q^{3-m},q^{k-3+m},q^m,q^{k-m})$ in \eqref{boulet-strict}. This is Berkovich-Uncu's method in spirit. Indeed, given a partition $\lambda\in \mathcal{D}$ with the $a,b,c,d$ fillings of the Ferrers' diagram as shown in the first diagram of Figure~\ref{2weights}, replacing each cell labelled by $a$ (resp. $b$, $c$, $d$) by $3-m$ (resp. $k-3+m$, $m$, $k-m$) cells in the Ferrers diagram of $\lambda$ sets up a bijection 
between $\mathcal{D}$ and $\mathcal{DS}_m^k$ .
\end{remark}

In what follows, 
if $\mathcal{F}$  (resp. $f$) is the set (resp. generating function) 
of partitions under certain constraints, then we use $\mathcal{F}_{N,M}$ (resp. $f_{N,M}$) to denote the subset (resp.  generating function) satisfying the extra conditions on the largest part ($\leq N$) and the number of parts ($\leq M$), where $N$ and $M$ are non-negative integers or $\infty$. By convention, we think of the empty partition as the only element in $\mathcal{F}_{N,M}$ if either $N=0$ or $M=0$. When $N=M=\infty$, we simply write $\mathcal{F}$ (resp.~$f$) instead of $\mathcal{F}_{\infty, \infty}$ (resp.~$f_{\infty, \infty}$).

A natural refinement of Theorem~\ref{decomp1} is to bound the largest part of the partitions. Given a partition $\pi\in \mathcal{S}^k$ with its largest part $\leq N$, suppose $\psi_k(\pi)=(\pi^1,\pi^2)$. Note that the decomposition $\psi_k$ dissects horizontally by rows, hence after the decomposition, the largest parts in both $\pi^1$ and $\pi^2$ are still bounded by the same number $N$. Namely, we prove the following bounded version of Theorem~\ref{decomp1}.

\begin{corollary}\label{decomp1N}
The restriction of the 
 map $\psi_{k,N}:=\psi_k|_{\mathcal{S}_{N,\infty}^k}: \pi \longmapsto (\pi^1, \pi^2)$
  is a weight-preserving bijection from $\mathcal{S}_{N,\infty}^k$ to
  $\mathcal{DS}_{N,\infty}^k\times \mathcal{E}_{N,\infty}^k$
such that $\ell(\pi)=\ell(\pi^1)+\ell(\pi^2)$ and
\begin{align}
\omega^k_{\pi}\big((a_i), (b_i)\big) &=\omega^k_{\pi^1}\big((a_i), (b_i)\big) 
\omega^k_{\pi^2}\big((a_i), (b_i)\big) .
\end{align} 
\end{corollary}

The parameters $(a_i),(b_i)$ in the $\omega^k$-weight encode modular information for odd-indexed parts as well as even-indexed parts. Once we specialize their values properly, we recover a handful of partition theorems of Rogers-Ramanujan type. We elaborate on this fruitful direction in the next two sections.

\section{Application to a companion of Capparelli's identities}\label{sec: mod3}
Sections~\ref{sec: mod3} and \ref{sec: mod2} can be viewed as immediate applications of our Theorem~\ref{main1}. For the $3$-strict case,  we shall consider both the infinite case $\mathcal{S}^3=\mathcal{S}^3_{\infty,\infty}$ and the bounded case $\mathcal{S}^3_{N,\infty}$. We need the following special case of Theorem~\ref{main1}.
\begin{theorem}\label{gf32strict} Let $R=abcdef$. 
\begin{align}
\label{gfE3}\sum\limits_{\pi\in \mathcal{E}^3}\omega^3_{\pi}(a,b,c,d,e,f) &=\dfrac{1}{(R;R)_{\infty}},\\
\label{gf3id}
\sum\limits_{\pi\in\mathcal{S}^3}\omega^3_{\pi}(a,b,c,d,e,f)&=\dfrac{(-a-ab,-abcd-abcde;R)_{\infty}}{(abc,R;R)_{\infty}}.
\end{align}
\end{theorem}

\subsection{Infinite Case}


Condition ii in Definition~\ref{level3} is equivalent to our definition for the set $\mathcal{DS}^3$, while condition i is checked to be equivalent to condition ii $+$ iii in our Definition~\ref{kAC}. 
Thus,  the $k=3$ case of  Theorem~\ref{bu-k} reduces to \eqref{bunewcomp}.
Moreover, we get the weighted generating function for $\mathcal{DS}^3$ upon combining Theorem~\ref{decomp1} for $k=3$, (\ref{gfE3}) and (\ref{gf3id}), and cancelling out the common factor $1/(R;R)_{\infty}$ from both sides.

\begin{corollary}\label{gfB} Let $R=abcdef$.
\begin{align}\label{gfBid}
\sum\limits_{\pi\in\mathcal{DS}^3}\omega^3_{\pi}(a,b,c,d,e,f)&=\dfrac{(-a-ab,-abcd-abcde;R)_{\infty}}{(abc;R)_{\infty}}.
\end{align}
\end{corollary}

\begin{Def} Given a partition $\pi$,
for $i=1,2$,  let $o_i(\pi)$ (resp. $e_i(\pi)$) be  the number of odd-indexed (resp. even-indexed) parts that are $\equiv i \pmod 3$. And we recall that $|\pi_o|$ (resp. $|\pi_e|$) is the sum of odd-indexed (resp. even-indexed) parts of $\pi$. 
\end{Def}
Then upon taking $a=sx, b=tx/s, c=x/t, d=uy, e=vy/u, f=y/v$ in (\ref{gfBid}) we get the following:
\begin{theorem}\label{gfBeva}
\begin{align}\label{gfBevaid}
\sum\limits_{\pi\in\mathcal{DS}^3}x^{|\pi_o|}y^{|\pi_e|}s^{o_1(\pi)}t^{o_2(\pi)}u^{e_1(\pi)}v^{e_2(\pi)}&=\dfrac{(-sx-tx^2,-ux^3y-vx^3y^2;x^3y^3)_{\infty}}{(x^3;x^3y^3)_{\infty}}.
\end{align}
\end{theorem}

Next we take $x=y=q$ in \eqref{gfBevaid} and consider two pairs of dual specializations, both of them can be interpreted as partition theorems.

\begin{theorem}\label{bunewcomp-refine}
For integers $n,i,j\geq 0, m\in\{1,2\}$, the number of partitions enumerated by $A_m(n)$ that have exactly $i$ parts $\equiv 2 \pmod 3$ and $j$ parts $\equiv 1 \pmod 3$ equals the number of partitions enumerated by $C_m(n)$ that have exactly $i$ parts $\equiv 3m-1 \pmod 6$ and $j$ parts $\equiv 3m+1 \pmod 6$.
\end{theorem}
\begin{proof}
When $m=1$, condition i becomes $\pi_{2i+1}\not\equiv 1 \pmod 3$ and $\pi_{2i}\not\equiv 2 \pmod 3$. Or equivalently, $o_1(\pi)=e_2(\pi)=0$. This means we should put $s=v=0$ in \eqref{gfBevaid} to get
\begin{align}
\label{s=v=0}
\sum\limits_{\pi\in\mathcal{DS}^3 \atop o_1(\pi)=e_2(\pi)=0}t^{o_2(\pi)}u^{e_1(\pi)}q^{|\pi|}&=\dfrac{(-tq^2,-uq^4;q^6)_{\infty}}{(q^3;q^6)_{\infty}}=(-tq^2,-q^3,-uq^4,-q^6;q^6)_{\infty},
\end{align}
and extract the coefficients of $t^iu^jq^n$ on both sides to prove the claim for $m=1$. And the case with $m=2$ means $\pi_{2i+1}\not\equiv 2 \pmod 3$ and $\pi_{2i}\not\equiv 1 \pmod 3$, which leads to putting  $t=u=0$ in \eqref{gfBevaid} to get
\begin{align}
\label{t=u=0}
\sum\limits_{\pi\in\mathcal{DS}^3 \atop o_2(\pi)=e_1(\pi)=0}s^{o_1(\pi)}v^{e_2(\pi)}q^{|\pi|}&=\dfrac{(-sq,-vq^5;q^6)_{\infty}}{(q^3;q^6)_{\infty}}=(-sq,-q^3,-vq^5,-q^6;q^6)_{\infty}.
\end{align}
Extracting coefficients of $s^jv^iq^n$ completes the proof.
\end{proof}
\begin{remark}
Theorem~\ref{bunewcomp-refine} refines \eqref{bunewcomp} and could also be derived from Theorem~2.5 in \cite{bu1} by sending $N$ to infinity. Indeed, the coefficient of
$t^iu^jq^n$ in the expansion of \eqref{s=v=0} is exactly $C_{1,\infty}(n,i,j)$ as in \cite{bu1}, while the coefficient of $s^iv^jq^n$ in the expansion of \eqref{t=u=0} is exactly $C_{2,\infty}(n,i,j)$. 
\end{remark}
The next theorem appears to be a new companion that cannot be deduced from the existing results.
\begin{theorem}\label{thm: FZ-inf}
For integers $n,i,j\geq 0, m\in\{1,2\}$, let $D_m^{\text{I}}(i,j,n)$ be 
the number of partitions of $n$ into distinct parts $\not\equiv -m \pmod 3$ that have exactly $i$ odd-indexed parts $\equiv m \pmod 3$ and $j$ even-indexed parts $\equiv m \pmod 3$, 
and $D_m^{\text{II}}(i,j,n)$  the number of partitions of $n$ into distinct parts $\not\equiv -m \pmod 3$ that have exactly $i$ parts $\equiv m \pmod 6$ and $j$ parts $\equiv m+3 \pmod 6$. Then
$$
D_m^{\text{I}}(i,j,n)=D_m^{\text{II}}(i,j,n).
$$
\end{theorem}
\begin{proof}
When $m=1$, no parts can be $\equiv -1\equiv 2 \pmod 3$ means that $o_2(\pi)=e_2(\pi)=0$, so we put $t=v=0$ in \eqref{gfBevaid} to get 
\begin{align}
\label{t=v=0}
\sum\limits_{\pi\in\mathcal{DS}^3 \atop o_2(\pi)=e_2(\pi)=0}s^{o_1(\pi)}u^{e_1(\pi)}q^{|\pi|}&=\dfrac{(-sq,-uq^4;q^6)_{\infty}}{(q^3;q^6)_{\infty}}=(-sq,-q^3,-uq^4,-q^6;q^6)_{\infty},
\end{align}
and compare coefficients of $s^iu^jq^n$ on both sides to get the claim. Similar arguments apply for $m=2$ upon putting $s=u=0$ in \eqref{gfBevaid} and getting 
\begin{align}
\label{s=u=0}
\sum\limits_{\pi\in\mathcal{DS}^3 \atop o_1(\pi)=e_1(\pi)=0}t^{o_2(\pi)}v^{e_2(\pi)}q^{|\pi|}&=\dfrac{(-tq^2,-vq^5;q^6)_{\infty}}{(q^3;q^6)_{\infty}}=(-tq^2,-q^3,-vq^5,-q^6;q^6)_{\infty}.
\end{align}
Comparing coefficients of $t^iv^jq^n$ on both sides completes the proof.
\end{proof}
\begin{example}
When $m=1, n=17$, we list out partitions of 
\begin{itemize}
\item type I: distinct, no parts $\equiv 2 \pmod 3$, $i$ odd-indexed parts $\equiv 1 \pmod 3$ and $j$ even-indexed parts $\equiv 1 \pmod 3$,
\item type II: distinct, no parts $\equiv 2 \pmod 3$, $i$ parts $\equiv 1 \pmod 6$, $j$ parts $\equiv 4 \pmod 6$.
\end{itemize}
Both types have the same count for each choice of $(i,j)$, as claimed by the last theorem.
\begin{table}[htbp]\caption{}
\centering
\begin{tabular}{|c|c|c|}
\hline
$(i,j)$ & type I & type II\\
\hline
(2,0) & (13,3,1), (10,6,1), (7,6,4) & (13,3,1), (9,7,1), (7,6,3,1)\\
\hline
(1,1) & (16,1), (13,4), (12,4,1), (10,7) & (16,1), (13,4), (12,4,1), (10,7)\\
 & (10,4,3), (9,7,1), (7,6,3,1) & (10,6,1), (9,4,3,1), (7,6,4)\\
\hline
(0,2) & (9,4,3,1) & (10,4,3)\\
\hline
\end{tabular}
\label{tab1}
\end{table}
\end{example}

\subsection{Bounded Case}
Similar to the infinite case, we find the generating functions for $\mathcal{E}_{N,\infty}^3$ and $\mathcal{S}_{N,\infty}^{3}$ first, then use them together with the $k=3$ case of Corollary~\ref{decomp1N} to deduce the bounded version of Corollary~\ref{gfB}. 

\begin{proposition} Let $R=abcdef$.
For any non-negative integer $N$,
\begin{align}
\label{gfE3b}\sum\limits_{\pi\in \mathcal{E}^3_{N,\infty}}\omega^3_{\pi}(a,b,c,d,e,f) &=\dfrac{1}{(R;R)_{\left\lfloor{N}/{3}\right\rfloor}}.
\end{align}
\end{proposition}
\begin{proof}
One thing to be noted is that, since partitions in $\mathcal{E}^3$ have only parts that are $\equiv 0 \pmod 3$, so one has $\mathcal{E}^3_{3l,\infty}=\mathcal{E}^3_{3l+1,\infty}=\mathcal{E}^3_{3l+2,\infty}$, which explains $\left\lfloor{N}/{3}\right\rfloor$. Otherwise the proof goes similarly as for (\ref{gfE3}).
\end{proof}

For any non-negative integer $N$ and $\mu\in \{0, 1,2\}$, we consider the generating function
\begin{align*}
S_{3N+\mu}^3 &:=S_{3N+\mu}^3(a,b,c,d,e,f):=\sum\limits_{\pi\in\mathcal{S}_{3N+\mu,\infty}^3}
\omega^3_{\pi},\\
DS_{3N+\mu}^3 &:=DS_{3N+\mu}^3(a,b,c,d,e,f):=\sum\limits_{\pi\in\mathcal{DS}_{3N+\mu,\infty}^3}
\omega^3_{\pi}.
\end{align*}
\begin{theorem} \label{gfS_3N012} Let $R=abcdef$.
We have 
\begin{align}
\label{gfDS3b3N}
S_{3N}^3 &=
\sum_{T}R^{\binom{t_1}{2}+\binom{t_2}{2}}F(T)
,\\
\label{gfDS3b3N+1}
S_{3N+1}^3 &=S_{3N}^3(a,b,c,d,e,f)+a(abc)^N
S_{3N}^3(d,e,f,a,b,c),\\
\label{gfDS3b3N+2}
S_{3N+2}^3 &=(1+a+ab)
\sum_{T}R^{\binom{t_1+1}{2}+\binom{t_2}{2}}F(T),\\
\label{gfDS}
DS_{3N+\mu}^3 &=(R;R)_N S_{3N+\mu}^3\quad\text{for}\quad  \mu\in \{0,1,2\},
\end{align}
where the summation $\sum_T$ is over all quadruples $T:=(t_1,t_2,t_3,t_4)\in \N^4$ such that
$$
\sum_{j=1}^4t_j=N, \text{ and }
F(T):=\dfrac{(a+ab)^{t_1}(abcd+abcde)^{t_2}(abc)^{t_3}}{(R;R)_{t_1}(R;R)_{t_2}(R;R)_{t_3}(R;R)_{t_4}}.
$$
\end{theorem}
\begin{proof}
Given $\pi\in\mathcal{S}^3_{3N,\infty}$, we can  decompose $\pi$ vertically
 into four types of blocks with width $3$ as in Fig.~\ref{sixtypes} and obtain a quadruple 
  $(\pi^{\sc I}, \pi^{\sc{II}}, \pi^{\sc{III}}, \pi^{\sc{IV}})$,  where $\pi^{\sc k}$ is 
  the partition obtained by assembling all the blocks of type $k$ in $\pi$. 
  Clearly the lengths of blocks of type I and II must be distinct, while those of type III and IV could be 
  repeated. Moreover the number of blocks $\pi^k$  is bounded by $N$. 
  Let $\mathcal{S}^3_{3N,\infty}(k)$ be the subset of partitions in $\mathcal{S}^3_{3N,\infty}$ whose blocks are exclusively of type $k$.  It is easy to 
 compute the generating functions   for   partitions in each $\mathcal{S}^3_{3N,\infty}(k)$ with a fixed  or bounded number of blocks.
\begin{itemize}
\item[I.] The generating function of partitions in $\mathcal{S}^3_{3N,\infty}(\text{I})$ with  
$t_1$ blocks of distinct lengths  is 
$$
\dfrac{(a+ab)^{t_1}}{(R;R)_{t_1}}R^{\binom{t_1}{2}}, \qquad 0\leq  t_1\leq N.
$$ 
\item[II.] The generating function of partitions in $\mathcal{S}^3_{3N,\infty}(\text{II})$ with  
$t_2$ blocks of distinct lengths  is 
$$
\dfrac{(abcd+abcde)^{t_2}}{(R;R)_{t_2}}R^{\binom{t_2}{2}}, \qquad 0\leq  t_2\leq N.
$$ 
\item[III.] The generating function of partitions in $\mathcal{S}^3_{3N,\infty}(\text{III})$ with  
$t_3$ blocks is 
$$
\dfrac{(abc)^{t_3}}{(R;R)_{t_3}}, \qquad 0\leq  t_3\leq N.
$$ 
\item[IV.] The generating function of partitions in $\mathcal{S}^3_{3N,\infty}(\text{IV})$ with  at most 
$t_4$ blocks  is 
$$
\dfrac{1}{(R;R)_{t_4}}, \qquad 0\leq  t_4\leq N.
$$ 
\end{itemize}
Putting all four types of blocks together leads to
 the constraint $\sum_{j=1}^4t_j=N$, in which case there are {\bf exactly} $t_1$ blocks of type I, $t_2$ blocks of type II, $t_3$ blocks of type III, and {\bf at most} $t_4$ blocks of type IV. 
 Thus the generating function of $\mathcal{S}^3_{3N,\infty}$ is
$$
\sum\limits_{t_1,t_2,t_3,t_4=0\atop t_1+t_2+t_3+t_4=N}^N \dfrac{(a+ab)^{t_1}R^{\binom{t_1}{2}}}{(R;R)_{t_1}}\dfrac{(abcd+abcde)^{t_2}R^{\binom{t_2}{2}}}{(R;R)_{t_2}}\dfrac{(abc)^{t_3}}{(R;R)_{t_3}}\dfrac{1}{(R;R)_{t_4}}\\{},
$$
which establishes \eqref{gfDS3b3N}. Next, for $S_{3N+1}^3$ and $S_{3N+2}^3$, the possibility of having $3N+1$ or $3N+2$ as the largest part will affect the generating function for $\mathcal{S}^3_{3N+\mu,\infty}(\text{I})$. More precisely, we have
\begin{itemize}
\item The generating function of partitions in $\mathcal{S}^3_{3N+1,\infty}(\text{I})$ with  
$t_1$ blocks of distinct lengths $\geq 3$ is composed of three parts, according to residue class of the largest part modulo $3$:
\begin{align*}
\text{largest part}\not\equiv 2 \pmod 3&: \quad (1+a)\dfrac{(a+ab)^{t_1}}{(R;R)_{t_1}}R^{\binom{t_1+1}{2}}, \qquad 0\leq  t_1\leq N.\\
\text{largest part}\equiv 2 \pmod 3 &: \quad ab\dfrac{(a+ab)^{t_1}}{(R;R)_{t_1}}R^{\binom{t_1+1}{2}}, \qquad 0\leq  t_1\leq N-1.
\end{align*}
Note that in the second case above, $t_1$ cannot be $N$, otherwise the largest part will be $3N+2>3N+1$. For the same reason we should have $0\leq t_1+t_2+t_3+t_4\leq N-1$ in this case. Therefore, instead of messing around with this change on the upper limit of the summation, we choose to analyse the largest part ($=3N+1$ or $\leq 3N$), and get \eqref{gfDS3b3N+1}, note the change of variables in the second term, due to the fact that the first part of the remaining parts is labelled as $d,e,f,\ldots$ in stead.
\item The generating function of partitions in $\mathcal{S}^3_{3N+2,\infty}(\text{I})$ with  
$t_1$ blocks of distinct lengths $\geq 3$ is 
$$
(1+a+ab)\dfrac{(a+ab)^{t_1}}{(R;R)_{t_1}}R^{\binom{t_1+1}{2}}, \qquad 0\leq  t_1\leq N,
$$ 
which leads to \eqref{gfDS3b3N+2}.
\end{itemize}
Finally, we combine \eqref{gfE3b}$\sim$\eqref{gfDS3b3N+2} with Corollary~\ref{decomp1N} to get \eqref{gfDS}.
\end{proof}

\begin{remark}\label{bdtoinfrmk}
To see the connection between the infinite case and the bounded case analytically, one simply applies one of the most elementary series-product identities due to Euler, for each finite sum in the formula. For instance, the following identity can be deduced from \cite[(2.2.6)]{andtp} by taking $t=a+ab, q=R$.
$$
\sum\limits_{t_1=0}^{\infty}\dfrac{(a+ab)^{t_1}R^{\binom{t_1}{2}}}{(R;R)_{t_1}}=(-a-ab;R)_{\infty}.
$$
\end{remark}

For nonnegative integers $(n_1, \ldots, n_m)$ suth that $N=n_1+\cdots+ n_m$, 
we define the $q$-multinomial coefficients:
$$
\qbin{N}{n_1,\ldots, n_m}{q}:=\dfrac{(q;q)_{N}}{ (q;q)_{n_1}\ldots(q;q)_{n_m}}.
$$
Let $\lceil x \rceil$ and $\lfloor x \rfloor$ be the usual ceiling and floor functions for $x\in \R$.
Now, in \eqref{gfDS}
with the same substitution for the parameters $(a, b, c, d, e,f)$ as in (\ref{gfBevaid}), namely 
$$
a=sx, \quad b=tx/s,  \quad c=x/t,  \quad d=uy,  \quad e=vy/u, \quad  f=y/v,
$$
 we arrive at the bounded version of Theorem~\ref{gfBeva}. 
\begin{theorem}\label{gfBeva-bounded}
 For $N\geq 0$ and $\mu\in \{0,1, 2\}$, define the generating function
 $$
 S_\mu:=S_\mu(s,t,u,v,x,y)=\sum\limits_{\pi\in\mathcal{DS}_{3N+\mu}^3}x^{|\pi_o|}y^{|\pi_e|}s^{o_1(\pi)}t^{o_2(\pi)}u^{e_1(\pi)}v^{e_2(\pi)}.
 $$
 Then,  we have 
\begin{align}\label{gfBevaid-bounded-02}
 S_\mu&= \left(1+\frac{\mu}{2}(sx+tx^2)\right)\sum_{(i,j,k,l)\in \N^4\atop i+j+k+l=N} \qbin{N}{i,j,k,l}{x^3y^3}(sx+tx^2)^{i}\\
\nonumber
&\hspace{4cm}\times (ux^3y+vx^3y^2)^{j}x^{3k}
(x^3y^3)^{\binom{i+{\mu}/{2}}{2}+\binom{j}{2}},\quad (\mu=0,2),\\
S_1&=S_0(s,t,u,v,x,y)+sx^{3N+1} S_0(u,v,s,t,y,x).\label{gfBevaid-bounded-1}
\end{align}
\end{theorem}

Next we take $x=y=q$ in \eqref{gfBevaid-bounded-02}, \eqref{gfBevaid-bounded-1} and 
define the generating function
\begin{align}\label{GF-bounded}
P_{3N+\mu}(s,t,u,v;q):
=\sum\limits_{\pi\in\mathcal{DS}_{3N+\mu}^3}s^{o_1(\pi)}t^{o_2(\pi)}u^{e_1(\pi)}v^{e_2(\pi)}q^{|\pi|},
\end{align}
i.e., the polynomial $S_\mu(s,t,u,v, q,q)$.
Similarly we consider the following four further specializations, corresponding to taking $s=v=0$, $t=u=0$, $t=v=0$ and $s=u=0$ respectively, and get the following theorem, which is a 
bounded version of \eqref{s=v=0} $\sim$ \eqref{s=u=0}, where
 the unspecified sums are over all quadruples $(i,j,k,l)\in \N^4$ such that $i+j+k+l=N$.

\begin{theorem}\label{bunewcomp-refine-bounded}
For $N\geq 0$, $\mu\in\{0,2\}$ we have:
\begin{align}
\label{s=v=0-b}
P_{3N+\mu}(0,t,u,0;q)&=
\left(1+\frac{\mu}{2} tq^2\right)\sum
\qbin{N}{i,j,k,l}{q^6}t^{i}u^{j}q^{3i^2+(3\mu-1)i+3j^2+j+3k},\\
\label{t=u=0-b}
P_{3N+\mu}(s,0,0,v;q)&=\left(1+\frac{\mu}{2} sq\right)\sum
\qbin{N}{i,j,k,l}{q^6}s^{i}v^{j}q^{3i^2+(3\mu-2)i+3j^2+2j+3k},\\
\label{t=v=0-b}
P_{3N+\mu}(s,0,u,0;q)&=\left(1+\frac{\mu}{2} sq\right)\sum
\qbin{N}{i,j,k,l}{q^6}s^{i}u^{j}q^{3i^2+(3\mu-2)i+3j^2+j+3k},\\
\label{s=u=0-b}
P_{3N+\mu}(0,t,0,v;q)&=\left(1+\frac{\mu}{2} tq^2\right)\sum
\qbin{N}{i,j,k,l}{q^6}t^{i}v^{j}q^{3i^2+\left(3\mu-1\right)i+3j^2+2j+3k},
\end{align}
and
%
\begin{align}
\label{s=v=0-b-1}
P_{3N+1}(0,t,u,0;q)&=\sum
\qbin{N}{i,j,k,l}{q^6}t^{i}u^{j}q^{3i^2-i+3j^2+j+3k},\\
\label{t=u=0-b-1}
P_{3N+1}(s,0,0,v;q)&=\sum\qbin{N}{i,j,k,l}{q^6}q^{3i^2-2i+3j^2+2j+3k}
\left(s^{i}v^{j}+s^{j+1}v^{i}q^{i-j+3N+1}\right),\\
\label{t=v=0-b-1}
P_{3N+1}(s,0,u,0;q)&=\sum\qbin{N}{i,j,k,l}{q^6}q^{3i^2-2i+3j^2+j+3k}\left(s^{i}u^{j}+s^{j+1}u^{i}q^{3N+1}\right),\\
\label{s=u=0-b-1}
P_{3N+1}(0,t,0,v;q)&=\sum
\qbin{N}{i,j,k,l}{q^6}t^{i}v^{j}q^{3i^2-i+3j^2+2j+3k}.
\end{align}
\end{theorem}

Recall the $q$-binomial coefficients
 $$
\qbin{n}{k}{q} :=
\begin{cases}\dfrac{(q;q)_n}{(q;q)_k(q;q)_{n-k}}, & \text{ for } n\geq k\geq 0,\\ 0, & \text{ otherwise. }\end{cases}
$$
For non-negative integers $N,\, i,\,j$ where $N\geq i,\,j$, and $m=1,\,2$, we define 
\begin{align}
\omega(m,i,j):&=(3i-m)i+(3j+m)j,\\
\pi(m,i,j):&=\omega(m,i,j)+(-1)^{m}i,
\end{align}
and
\begin{align}
F_{N}(i,j; q):&={N\brack i, j,N-i-j}_{q^6} (-q^3;q^3)_{N-i-j},\\
G_{N}(i,j; q):&=\frac{1-q^{3(N+1+i- j)}}{1-q^{6(N+1)}}  F_{N+1}(i,j; q).
\end{align}
\begin{theorem} Let $[x^iy^j]p(x,y)$ be the coefficient of $x^iy^j$ in the polynomial $p(x,y)$. 
For non-negative integers $N,\, i,\,j$ where $N\geq i,\,j$, and $\mu=0, \,1,\,2$, we  have 
\begin{align}
[t^iu^j]P_{3N+\mu}(0,t,u,0;q)&=q^{\omega(1,i,j)}\left((\delta_{0\mu}+\delta_{1\mu}) F_{N}(i,j; q)+
\delta_{2\mu} G_{N}(i,j; q)\right), \label{coeff1}\\
[s^iv^j]P_{3N+\mu}(s,0,0,v;q)&=q^{\omega(2,i,j)}\left(\delta_{0\mu}F_{N}(i,j; q)+
(\delta_{1\mu}+\delta_{2\mu}) G_{N}(i,j; q)\right), \label{coeff2}\\
[s^iu^j]P_{3N+\mu}(s, 0,u,0;q)&=q^{\pi(1,i,j)}\left(\delta_{0\mu} F_{N}(i,j; q)+
(\delta_{1\mu}+\delta_{2\mu}) G_{N}(i,j; q)\right), \label{coeff3}\\
[t^iv^j]P_{3N+\mu}(0,t,0,v;q)&=q^{\pi(2,i,j)}\left((\delta_{0\mu}+\delta_{1\mu}) F_{N}(i,j; q)+
\delta_{2\mu} G_{N}(i,j; q)\right).\label{coeff4}
\end{align} 
\end{theorem}
\begin{proof}
By extracting the corresponding coefficients in 
\eqref{s=v=0-b}-\eqref{s=u=0-b} (resp. \eqref{s=v=0-b-1}-\eqref{s=u=0-b-1}) and applying the following well known identity(see \cite[Page 49]{andtp}):
\begin{align}\label{rogers}
\sum\limits_{j=0}^{n}\qbin{n}{j}{q^2}q^j=(-q;q)_n,
\end{align}
we recover the right hand sides of the above identities. We omit the details.
\end{proof}
The first two identities \eqref{coeff1} and \eqref{coeff2} 
 are equivalent to the main Theorem~2.2 in \cite{bu1}, which leads to their Theorem~2.5 in \cite{bu1}
 for $\mu=0$,
although one needs some careful verification on the different-looking boundary conditions to see this equivalence from the left hand side.  In the same vein, from  the $\mu=0$ case of 
\eqref{coeff3} and \eqref{coeff4}  
 we can
derive a refinement 
of Theorem~\ref{thm: FZ-inf}.

\begin{Def}
For integers $N,n,i,j\geq 0, m\in\{1,2\}$, let $D_{m,3N}^{\text{I}}(i,j,n)$ be the number of partitions of $n$ into distinct parts such that
\begin{enumerate}[i.]
\item each part $\not\equiv -m \pmod 3$;
\item each part $\leq 3N$;
\item there are exactly $i$ odd-indexed parts $\equiv m \pmod 3$;
\item there are exactly $j$ even-indexed parts $\equiv m \pmod 3$.
\end{enumerate}
Let $D_{m,3N}^{\text{II}}(i,j,n)$ be the number of partitions of $n$ into distinct parts such that
\begin{enumerate}[i.]
\item each part $\not\equiv -m \pmod 3$;
\item there are exactly $i$ parts $\equiv m \pmod 6$ and these parts are all $\leq 6N+m-6$;
\item there are exactly $j$ parts $\equiv m+3 \pmod 6$ and these parts are all $\leq 6(N-i)+m-3$;
\item all parts $\equiv 0 \pmod 3$ are $\leq 3(N-i-j)$.
\end{enumerate}
\end{Def}

\begin{theorem}[$\leq 3N$ version of Theorem~\ref{thm: FZ-inf}]\label{thm: FZ-finite}
For integers $N,n,i,j\geq 0, m\in\{1,2\}$,
$$
D_{m,3N}^{\text{I}}(i,j,n)=D_{m,3N}^{\text{II}}(i,j,n).
$$
\end{theorem}
\begin{proof}
The proof is analogous to the proof of Theorem~2.5 in \cite{bu1}, we sketch the case of $m=1$ for completeness. When $m=1$, it is clear that the generating function of $D_{1,3N}^{\text{I}}(i,j,n)$ is given by the left hand side of \eqref{coeff3} with $\mu=0$, in which case its right hand side becomes
\begin{align*}
q^{\pi(1,i,j)}F_N(i,j;q)=\left(q^{6\binom{i+1}{2}-5i}\qbin{N}{i}{q^6}\right)\left(q^{6\binom{j+1}{2}-2j}\qbin{N-i}{j}{q^6}\right)(-q^3;q^3)_{N-i-j}.
\end{align*}
It is evident that $q^{6\binom{i+1}{2}-5i}\qbin{N}{i}{q^6}$ generates exactly $i$ distinct parts $\equiv 1 \pmod 6$, and each part $\leq 6N-5$, matching condition ii in the definition of $D_{1,3N}^{\text{II}}(i,j,n)$. Similarly, $q^{6\binom{j+1}{2}-2j}\qbin{N-i}{j}{q^6}$ generates exactly $j$ distinct parts $\equiv 4 \pmod 6$, each part $\leq 6(N-i)-2$, matching condition iii. Lastly, $(-q^3;q^3)_{N-i-j}$ generates distinct parts that are $\equiv 0 \pmod 3$, each $\leq 3(N-i-j)$, matching condition iv. Collectively, we have condition i, hence we see the right hand side is exactly the generating function of $D_{1,3N}^{\text{II}}(i,j,n)$, this completes the proof of $D_{1,3N}^{\text{I}}(i,j,n)=D_{1,3N}^{\text{II}}(i,j,n)$. The proof of $D_{2,3N}^{\text{I}}(i,j,n)=D_{2,3N}^{\text{II}}(i,j,n)$ follows from similar verification and thus omitted.
\end{proof}

\begin{remark}
Note that Berkovich-Uncu's Theorem~2.5 and our Theorem~\ref{thm: FZ-finite}
only give partition interpretation for $\mu=0$ case of the right-hand sides of \eqref{coeff1}-\eqref{coeff4} and 
for the $\mu=1$ case of \eqref{coeff1} and \eqref{coeff4}, 
 while other  
cases when  $\mu=1,2$ were missed. It would be interesting to find out whether there are 
any  combinatorial interpretations for  these missing cases.
\end{remark}
%

\section{Application to bounded versions of Boulet's generating functions}{\label{sec: mod2}}

In this section, we develop along the same line as in last section for the new case of $k=2$, with Boulet's weight $\omega^2_\pi(a,b,c,d)$, see \eqref{boulet}. First off we record the special $k=2$ case of Theorem~\ref{main1}, which gives Boulet's \eqref{boulet-strict} as a quick corollary.
\begin{theorem}\label{gf2strict} Let $Q=abcd$.
\begin{align}
\label{gfE2}\sum\limits_{\lambda\in \mathcal{E}^2}\omega^2_{\lambda}(a,b,c,d) &=\dfrac{1}{(Q;Q)_{\infty}},\\
\label{gf2id}
\sum\limits_{\pi\in\mathcal{S}^2}\omega^2_{\pi}(a,b,c,d)&=\dfrac{(-a,-abc;Q)_{\infty}}{(ab,Q;Q)_{\infty}}.
\end{align}
\end{theorem}

Since the constructions are highly analogous, most proofs are either omitted or 
sketched, to avoid unnecessary repetitions. And since we have made quite a few observations 
(see identities (\ref{boulet}) through (\ref{boulet-stric-bu})) for the infinite case in the introduction, we begin here with the bounded case. 

\subsection{Single-bounded Case}\label{BC1}
The main goal of this subsection is to give a new proof of (\ref{iz1}) using $2$-strict partitions. First recall that $\mathcal{DS}^2=\mathcal{D}$. In view of Corollary~\ref{decomp1N}, to get the generating function for $\mathcal{D}_{N,\infty}$, it suffices to find the generating functions for $\mathcal{E}^2_{N,\infty}$ and $\mathcal{S}^2_{N,\infty}$ separately.

\begin{proposition} Let $Q=abcd$.
For any non-negative integer $N$,
\begin{align}
\label{gfE2b}\sum\limits_{\pi\in \mathcal{E}^2_{N,\infty}}\omega^2_{\pi}(a,b,c,d) &=\dfrac{1}{(Q;Q)_{\left\lfloor{N}/{2}\right\rfloor}}.
\end{align}
\end{proposition}

\begin{theorem}\label{gfS2b} Let $Q=abcd$.
For any non-negative integer $N$, $\nu\in\{0,1\}$,
\begin{align}\label{gfS2bid}
\sum\limits_{\pi\in\mathcal{S}^2_{2N+\nu,\infty}}\omega^2_{\pi}(a,b,c,d)&=\sum\limits_{i=0}^{N}\dfrac{(-a;Q)_{N-i+\nu}(-c;Q)_{i}(ab)^i}{(Q;Q)_{N-i}(Q;Q)_{i}}\\
\nonumber &=\dfrac{1}{(Q;Q)_{N}}\sum\limits_{i=0}^{N}\qbin{N}{i}{Q}(-a;Q)_{N-i+\nu}(-c;Q)_{i}(ab)^i.
\end{align}
\end{theorem}
\begin{proof}
Given any $\pi\in\mathcal{S}^2_{2N+\nu,\infty}$, since odd parts of $\pi$ are all distinct, for the conjugate $\pi'$ we must have
\begin{align}\label{gap<2}
\pi_{2i-1}'-\pi_{2i}' &= 0 \text{ or } 1,\; \forall i=1,2,\ldots \; .
\end{align}
This will guarantee that when we decompose $\pi$ into blocks of width $2$, $\{(\pi_1',\pi_2'),(\pi_3',\pi_4'),\ldots\}$, then we only have the following four types (filled with $\omega^2$-label):
$$
\begin{array}{ccccccc}
a \quad b & \qquad & a \quad  b & \qquad & a\quad  b & \qquad & a \quad  b\\
c \quad d & \qquad & c \quad  d & \qquad & c \quad  d & \qquad & c \quad  d\\
\vdots \quad \vdots & \qquad & \vdots \quad  \vdots & \qquad &\vdots \quad \vdots & \qquad & \vdots \quad  \vdots\\
a \quad  & \qquad & a \quad b & \qquad & a \quad  b & \qquad & a \quad  b\\
 \quad & \qquad & c \quad \phantom{b} & \qquad &\quad   & \qquad & c \quad  d\\
 &&&&&&\\
\hspace{10pt}\text{I } & \qquad & \hspace{10pt}\text{II} &    \qquad &\hspace{10pt} \text{III} &    \qquad & \hspace{10pt}\text{IV}   \\
\end{array}
$$
As for the constraint that the largest part of $\pi\leq 2N+\nu$, when $\nu=0$, the total number of blocks in all four types is $\leq N$; when $\nu=1$, the total number of blocks in all four types is $\leq N+1$, with equality only when the largest part is exactly $2N+1$ and hence there is a block of type I with a single cell labelled $a$. And also note that the blocks of type III and IV could be repeated while type I and II must all be distinct. The above analysis amounts to produce (\ref{gfS2bid}). These generating functions come most naturally by considering rows of the labelled Ferrers' diagram, even though the block types are defined by columns. Indeed, fix $0\leq i\leq N$,  the series
$$\dfrac{(-c;Q)_i(ab)^i}{(Q;Q)_i}
$$ generates exactly $i$ blocks of type II (distinct) and III, while the series
$$\dfrac{(-a;Q)_{N-i+\nu}}{(Q;Q)_{N-i}}
$$ accounts for at most $N-i+\nu$ blocks of type I (distinct) and IV, where the extra block when $v=1$ can be only of type I. 
\end{proof}

In \cite{iz}, Ishikawa and the second author considered the bounded version of Boulet's 
formula and obtained the following series expansion via application of results on associated Al-Salam-Chihara polynomials.
\begin{corollary}[\cite{iz}] Let $\nu\in\{0,1\}$. Then
\begin{align}
\label{iz1} \Psi_{2N+\nu,\infty}(a,b,c,d) &= \sum\limits_{i=0}^{N}\qbin{N}{i}{Q}(-a;Q)_{N-i+\nu}(-c;Q)_{i}(ab)^i, \\
\label{iz2} \Phi_{2N+\nu,\infty}(a,b,c,d) &= \dfrac{1}{(ac;Q)_{N+\nu}}\sum\limits_{i=0}^{N}\dfrac{(-a;Q)_{N-i+\nu}(-c;Q)_{i}(ab)^i}{(Q;Q)_{N-i}(Q;Q)_{i}}.
\end{align}

\end{corollary}
\begin{proof}
To get (\ref{iz1}), one simply combines (\ref{gfS2bid}) with Corollary~\ref{decomp1N} (case $k=2$) and (\ref{gfE2b}), and then cancel out the common factor $1/(Q;Q)_{\left\lfloor{N}/{2}\right\rfloor}$. The connection between the strict partition case (\ref{iz1}) and the ordinary partition case (\ref{iz2}) has already been noticed several times thus omitted here, see for example \cite[Theorem 4.1]{iz}, see also \cite{bu2,bou}.
\end{proof}

\begin{remark}
Extracting the coefficient of $t^iz^j$ in $\Psi_{2N+\nu,\infty}(qt,q/t,qz,q/z)$ using \eqref{iz1} and \eqref{rogers} yields Theorem~2.1 of 
\cite{bu2}.
Berkovich and Uncu~\cite[Theorem 6.1]{bu2} finitized Boulet's construction \cite{bou} to get:
\begin{align}
\label{bu1} \Psi_{2N+\nu,\infty}(a,b,c,d) &= \sum\limits_{i=0}^{N}\qbin{N}{i}{Q}(-a;Q)_{i+\nu}(-abc;Q)_{i}\dfrac{(ac;Q)_{N+\nu}}{(ac;Q)_{i+\nu}}(ab)^{N-i}, \\
\label{bu2} \Phi_{2N+\nu,\infty}(a,b,c,d) &= \sum\limits_{i=0}^{N}\dfrac{(-a;Q)_{i+\nu}(-abc;Q)_i(ab)^{N-i}}{(Q;Q)_{i}(ac;Q)_{i+\nu}(Q;Q)_{N-i}}.
\end{align}
They remarked that the transition from (\ref{bu1}, \ref{bu2}) to (\ref{iz1}, \ref{iz2}) requires a ${}_3\phi_1$ to ${}_2\phi_1$ transformation \cite[(III.8)]{gara}. We note that Andrews also noticed this ``two expansions for one function'' phenomenon in the specialized case at the end of his paper \cite{and2}.
\end{remark}

Next proposition highlights the role played by conjugation in our study.
\begin{proposition}
For $N$ and $M$ being any positive integers or $\infty$, the operation of conjugation, denoted as $\tau$, is a bijection from $\mathcal{P}_{N,M}$ to $\mathcal{P}_{M,N}$, such that for any $\pi\in\mathcal{P}_{N,M}$, we have 
$$\omega^2_{\pi}(a,b,c,d)=\omega^2_{\tau(\pi)}(a,c,b,d).
$$
 In terms of generating function, we have
\begin{align}\label{conju}
\Phi_{N,M}(a,b,c,d) &=\Phi_{M,N}(a,c, b,d).
\end{align}
\end{proposition}
\begin{proof}
Simply note that after applying conjugation, a partition in $\mathcal{P}_{N,M}$ becomes a partition in $\mathcal{P}_{M,N}$, and the labelling on the Ferrers' diagram now becomes $\{a,c,a,c,\ldots\}$ in the odd-indexed rows, and $\{b,d,b,d,\ldots\}$ in the even-indexed rows.
\end{proof}
As an immediate application, we could apply \eqref{conju} to derive two different expansions of the generating function for $\Phi_{\infty, 2M+\mu}(a,b,c,d)$ that are equivalent to (\ref{iz2}) and (\ref{bu2}) respectively, we leave it as an exercise for the interested readers.

We go on with some further observations that hopefully clarify the mystery around identities (\ref{iz1}, \ref{iz2}) and (\ref{bu1}, \ref{bu2}).

First off, we would like to remark that our definition of $2$-strict partitions and the way of dissecting $2$-strict partitions vertically into blocks of width $2$, can be viewed as a natural dual of Boulet's construction in \cite{bou} to prove (\ref{boulet}) and (\ref{boulet-strict}). More precisely, for each $\omega^2$-labelled Ferrers' diagram, we remove pairs of {\it rows} with odd length to get $2$-strict partitions and then read {\it columns} of these partitions by pairs, while Boulet's approach was to remove pairs of {\it columns} with odd length and read {\it rows} by pairs. The connection is clearly established via conjugation. 

Secondly, we note that Yee's \cite{yee} and Sills' \cite{sil} methods, when interpreted appropriately, are equivalent to ours in the special case 
$$(a,b,c,d)=(yzq,yq/z,zq/y,q/yz).
$$ 
Moreover, with this comparison in mind, it now becomes clear why we should have two different expansions (\ref{iz2}) and (\ref{bu2}) (resp. (\ref{iz1}) and (\ref{bu1})) for the same function 
$$\Phi_{2N+\nu,\infty}(a,b,c,d)\quad  (\text{resp.}\quad  \Psi_{2N+\nu,\infty}(a,b,c,d)). 
$$
The  reason is that  the constraint ($2N+\nu,\infty$) is asymmetric with respect to conjugation. But when we consider the infinite case $\Phi(a,b,c,d)$ (resp. $\Psi(a,b,c,d)$) or the doubly-bounded case $\Phi_{N,M}(a,b,c,d)$ (resp. $\Psi_{N,M}(a,b,c,d)$), these two dual approaches will only lead to essentially one expansion. Indeed, the two seemingly different expansions will become identical upon change of variables. We shall give more details in next subsection.

\subsection{Doubly-bounded Case}
Naturally, the next step is to consider the most general functions $\Phi_{N,M}(a,b,c,d)$ and $\Psi_{N,M}(a,b,c,d)$. We note that in an unpublished work, Chen, Lai and Wu \cite{clw} analytically obtained the expansions for all four cases of $\Phi_{N,M}(a,b,c,d)$ according to parity of $N$ and $M$, namely, 
$(N,M)\equiv (0,0),(0,1),(1,0),(1,1) \pmod 2$.
We note that Berkovich and Uncu~\cite[Theorems 6.2 and 6.3]{bu2} also obtained 
similar expansions for $\Phi_{N,M}$. 
%

We first give the doubly-bounded version of Theorem~\ref{gf2strict}, and then we can state the unified expansion for $\Phi_{N,M}(a,b,c,d)$, with another identity connecting $\Psi_{N,M}(a,b,c,d)$ as well.
\begin{theorem}\label{bound-s2}
Given any non-negative integers $N,M$, we have:
\begin{align}
\label{gfS2n0m1}
\sum_{\pi\in\mathcal{S}^2_{2N,2M+1}}\omega^2_{\pi}(a,b,c,d) &= \sum\qbin{M+1}{m_1}{Q}a^{m_1}Q^{\binom{m_1}{2}}\\
\nonumber &\times\qbin{M+m_2}{m_2}{Q}(ab)^{m_2} \qbin{M}{m_3}{Q}(abc)^{m_3}Q^{\binom{m_3}{2}}\qbin{M+m_4}{m_4}{Q},\\
\label{gfS2n1m1}
\sum_{\pi\in\mathcal{S}^2_{2N+1,2M+1}}\omega^2_{\pi}(a,b,c,d) &= \sum
\qbin{M}{m_1}{Q}(1+a)a^{m_1}Q^{\binom{m_1+1}{2}}\\
\nonumber &\times\qbin{M+m_2}{m_2}{Q}(ab)^{m_2} \qbin{M}{m_3}{Q}(abc)^{m_3}Q^{\binom{m_3}{2}}\qbin{M+m_4}{m_4}{Q},\\
\label{gfS2n0m0} 
\sum_{\pi\in\mathcal{S}^2_{2N,2M}}\omega^2_{\pi}(a,b,c,d) &= \sum\qbin{M}{m_1}{Q}a^{m_1}Q^{\binom{m_1}{2}}\\
\nonumber &\times\qbin{M+m_2-1}{m_2}{Q}(ab)^{m_2}\qbin{M}{m_3}{Q}(abc)^{m_3}Q^{\binom{m_3}{2}}\qbin{M+m_4}{m_4}{Q},
\end{align}
\begin{align}
\label{gfS2n1m0} 
\sum_{\pi\in\mathcal{S}^2_{2N+1,2M}}\omega^2_{\pi}(a,b,c,d) &= \sum
\qbin{M-1}{m_1}{Q}(1+a)a^{m_1}Q^{\binom{m_1+1}{2}}\\
\nonumber &\times\qbin{M+m_2-1}{m_2}{Q}(ab)^{m_2}\qbin{M}{m_3}{Q}(abc)^{m_3}Q^{\binom{m_3}{2}}\qbin{M+m_4}{m_4}{Q},
\end{align}
where the implicit sums are over all the quadruples $(m_1,m_2,m_3,m_4)\in \N^4$ satisfying
$m_1+m_2+m_3+m_4 = N$.
\end{theorem}

\begin{proof}
Similar arguments as in the proof of Theorem~\ref{gf32strict} and Theorem~\ref{gfS2b} can be applied to get (\ref{gfS2n0m1})$\sim$(\ref{gfS2n1m0}). The new constraint on the number of parts is reflected in the expression as $q$-binomial coefficients. For instance, in (\ref{gfS2n0m1}) $\qbin{M+1}{m_1}{Q}a^{m_1}Q^{\binom{m_1}{2}}$ replaces $(-a;Q)_{\infty}$ for generating type I blocks, and $\qbin{M+m_2}{m_2}{Q}(ab)^{m_2}$ replaces $\dfrac{1}{(ab;Q)_{\infty}}$ for type II and so on.  And one needs to make some extra effort to take care of different cases corresponding to the parity of the constraints. The details are omitted.
\end{proof}

The following result gives the explicit formulae for the bounded versions of
both $\Phi$ and $\Psi$ as multiple sums.
\begin{theorem}
\label{fznm}
For $N,M$ being non-negative integers, $\nu,\mu=0 \text{ or } 1$ such that $N+\nu\geq 1$, we have the following expansions:
\begin{align}\label{phinm}
&\phantom{=\;}\Phi_{2N+\nu,2M+\mu}(a,b,c,d) \\
\nonumber &=\delta_{0\mu}(ac)^M\qbin{N+M+\nu-1}{M}{Q}+\sum\limits_{k=0}^{M+\mu-1}(ac)^k\qbin{N+k+\nu-1}{k}{Q}\\
\nonumber &\times \sum\limits_{m_1,m_2,m_3,m_4\geq 0\atop m_1+m_2+m_3+m_4 = N}^{N}\qbin{M-k+m_4}{m_4}{Q}\qbin{M-k+\mu-\nu}{m_1}{Q}(1+a\nu)a^{m_1}Q^{\binom{m_1+\nu}{2}}\\
\nonumber &\times\qbin{M-k+\mu-1+m_2}{m_2}{Q}(ab)^{m_2}\qbin{M-k}{m_3}{Q}(abc)^{m_3}Q^{\binom{m_3}{2}},\\
\label{psinm}
&\phantom{=\;}\Psi_{N,M}(a,b,c,d) \\
\nonumber &=\sum_{m=0}^{\lfloor M/2\rfloor}(-1)^m
 \sum_{k=0}^m {{\lfloor N/2\rfloor} \brack k}_Q{{\lceil N/2\rceil} \brack m-k}_Q(ac)^{m-k} Q^{k(k+1-m)+{m\choose 2}}\Phi_{N,M-2m}(a,b,c,d).
\end{align}
\end{theorem}
\begin{proof}
To show \eqref{phinm}, we note that for a given unrestricted partition $\pi\in\mathcal{P}_{2N+\nu,2M+\mu}$, we can repeatedly remove even copies of parts with odd length and still keep the $\omega$-label as we showed in the proof of Theorem~\ref{decomp1}. The remaining partition is now $2$-strict, which we can use 
Theorem~\ref{bound-s2} to deal with. On the other hand, the removed parts (say $k$ pairs of odd parts) are generated by $(ac)^k\qbin{N+k+\nu-1}{k}{Q}$. The first term $\delta_{0\mu}(ac)^M\qbin{N+M+\nu-1}{M}{Q}$ accounts for the special case when $\mu=0$ and there is nothing remaining after our removal process.

Note that
\begin{align}
\frac{\sum_{\pi\in\mathcal{P}_{N,\infty}}\omega^2_{\pi}(a,b,c,d)z^{\ell(\pi)}}{1-z}&=\sum_{M\geq 0} \Phi_{N,M}(a,b,c,d)z^M,\label{key1}\\
\frac{\sum_{\pi\in\mathcal{D}_{N,\infty}}\omega^2_{\pi}(a,b,c,d)z^{\ell(\pi)}}{1-z}&=\sum_{M\geq 0} \Psi_{N,M}(a,b,c,d)z^M.\label{key2}
\end{align}
It follows from the connection formula \cite[Theorem 4.1]{iz} between 
the numerators of 
the left-hand sides of \eqref{key1} and \eqref{key2}
that  
\begin{align}\label{phipsi}
\sum_{M\geq 0} \Phi_{N,M}(a,b,c,d)z^M=
\frac{\sum_{M\geq 0} \Psi_{N,M}(a,b,c,d)z^M}{(z^2 Q;Q)_{\lfloor N/2\rfloor}(z^2 ac;Q)_{\lceil N/2\rceil}}.
\end{align}
Using the known identity~\cite[(3.3.6)]{andtp} 
\begin{align}\label{euler1}
(z;q)_N=\sum_{j=0}^N{N\brack j}_q (-1)^jz^j q^{j(j-1)/2},
\end{align}
we derive \eqref{psinm} from \eqref{phipsi}.
\end{proof}


\begin{proposition}\label{lem} We have 
\begin{align}\label{lem1}
[s^0t^j]\Phi_{2N+\nu, 2M+\mu}&(qs,q/s,qt,q/t)=q^{j(2j+1)}{M\brack j}_{q^4}{2M+N+\mu-j\brack N-j}_{q^2},\\
\label{lem2}
[s^it^0]\Phi_{2N+\nu, 2M+\mu}&(qs,q/s,qt,q/t)\\
&=
q^{i+2(i+\nu)(i+\nu-1)}{M+\mu-\nu\brack i}_{q^4}
{2M+\mu+N-i\brack N-i}_{q^2}\nonumber \\
&+\nu q^{i+2(i+\nu-1)(i+\nu-2)}{M+\mu-\nu\brack i-1}_{q^4}{2M+\mu+N-i+1\brack N-i+1}_{q^2}.\nonumber
\end{align}
\end{proposition}
\begin{proof} 
Recall~\cite[p. 36]{andtp} that
\begin{align}\label{euler2}
\frac{1}{(z;q)_{N}}=\sum_{j=0}^\infty {N-1+j\brack j}_q z^j.
\end{align}
Since  $(z;q^2)_{N+1} (zq;q^2)_{N+\mu}=(z;q)_{2N+1+\mu}$ for $\mu\in \{0,1\}$,  it follows from \eqref{euler2} that
\begin{align}\label{conv2}
\sum_{i+j=L}{N+i\brack i}_{q^2}   {N+\mu-1+j\brack j}_{q^2}q^j= {2N+\mu+L\brack L}_{q}.
\end{align}
Now, making the previous  substitution  $(a,b,c,d)=(qs,q/s,qt,q/t)$ in \eqref{phinm}, 
we have $Q=q^4$, $ac=q^2st$ and $abc=q^3t$. Hence, extracting the coefficient of $s^0t^j$ (resp. $s^it^0$), i.e., 
setting $s=0$ and then extracting the coefficient of $t^j$ (resp. setting $t=0$ and then extracting the coefficient of $s^i$), 
 we obtain \eqref{lem1} (resp. \eqref{lem2}). Indeed, by  \eqref{phinm} we have 
 \begin{align*}
 &[s^0t^j]\Phi_{2N+\nu, 2M+\mu}(qs,q/s,qt,q/t)\\
 &=[t^j]\sum_{m_2+m_3+m_4=N} {M+m_4\brack m_4}_{q^4}
 {M+\mu-1+m_2\brack m_2}_{q^4} q^{2m_2}{M\brack m_3}_{q^4}q^{m_3(2m_3+1)} t^{m_3}\\
 &=q^{j(2j+1)}{M\brack j}_{q^4}\sum_{m_2+m_4=N-j} {M+m_4\brack m_4}_{q^4}
 {M+\mu-1+m_2\brack m_2}_{q^4} q^{2m_2}.
 \end{align*}
 This yields \eqref{lem1} by applying \eqref{conv2}.
\end{proof}

Let ${P}_N(i,j,m,q)$ be the generating function for the number of 
ordinary partitions with largest  part $N$ with $i$ 
odd-indexed, $j$ even-indexed odd parts and  at most $m$ even parts. It's not difficult to see that
\begin{align}\label{key-formula}
{P}_N(i,j,m,q)=[s^i t^j]\left(\Phi_{N,m+i+j}(qs,q/s,qt,q/t)-\Phi_{N-1,m+i+j}(qs,q/s,qt,q/t)  \right).
\end{align}
From the above proposition we derive immediately the following explicit formulae by utilizing the well known Pascal-like relations for the $q$-binomial coefficients, see \cite[(3.3.3)-(3.3.4)]{andtp}.
\begin{corollary}\label{ana-final-bu}
Let $i, j,m$ and $N$ be non-negative integers. Then,
\begin{align}
{P}_{2N}(0,j,m,q)&=q^{2N+j(2j-1)}{\lfloor\frac{m+j}{2} \rfloor \brack j}_{q^4}{m+N-1\brack N-j}_{q^2},\\
{P}_{2N}(i,0,m,q)&=q^{2N+i(2i+1)}{\lceil \frac{m+i}{2}\rceil-1\brack i}_{q^4}
{m+N-1\brack N-i}_{q^2},\nonumber \\
{P}_{2N+1}(i,0,m,q)&=q^{2N+i(2i-3)+2}{\lceil \frac{m+i}{2}\rceil-1\brack i-1}_{q^4}
{m+N\brack N-i+1}_{q^2},\quad \text{for }i\geq 1.\nonumber
\end{align}
\end{corollary}

\begin{remark} The above result is comparable to 
Berkovich-Uncu's result for the coefficient of $s^i t^j$ where  $i=0$ or $j=0$  in
$$
\Psi_{2N+\nu,m+i+j}(qs,q/s,qt,q/t)-\Psi_{2N+\nu,m+i+j-1}(qs,q/s,qt,q/t);
$$
 see \cite[Proposition~7.4 ]{bu2}, where it has been stated without proof. It should be possible to prove their result similarly as Corallary~\ref{ana-final-bu} by applying Proposition~\ref{lem} and \eqref{phipsi}, we leave it to the motivated readers. Instead we supply a bijective proof below for completeness.
\end{remark}

\begin{proposition}[Berkovich-Uncu\cite{bu2}]\label{final-bu}
Let $\tilde{P}_N(i,j,m,q)$ be the generating function for the number of partitions into distinct parts $\leq N$ with $i$ odd-indexed, $j$ even-indexed odd parts and $m$ even parts. Then we have, for $\nu\in\{0,1\}$,
\begin{align}\label{final-bu-id}
\tilde{P}_{2N+\nu}(0,j,m,q)&=q^{j(j+1)+m(m+1)-j(-1)^{m+j}}{\lfloor \frac{m+j}{2}\rfloor\brack j}_{q^4}{N+j\brack j+m}_{q^2},\\
\nonumber
\tilde{P}_{2N+1}(i,0,m,q)&=q^{i(i+1)+m(m+1)+i(-1)^{m+i}}{\lceil \frac{m+i}{2}\rceil\brack i}_{q^4}{N+i\brack i+m}_{q^2},\\
\nonumber
\tilde{P}_{2N}(i,0,m,q)&=q^{i(i+1)+m(m+1)+i(-1)^{m+i}}{\lceil \frac{m+i}{2}\rceil\brack i}_{q^4}{N+i-1\brack i+m}_{q^2}\\
\nonumber
&\phantom{qqq}+q^{i(i+1)+m(m-1)+i(-1)^{m+i}+2N}{\lfloor \frac{m+i-1}{2}\rfloor\brack i}_{q^4}{N+i-1\brack i+m-1}_{q^2}.
\end{align}
\end{proposition}

\begin{proof}
We first show the formula for $\tilde{P}_{2N+\nu}(0,j,m,q)$, and let us begin with the case when $m+j$ is even. First note that, since $i=0$, all odd parts have to be even-indexed, hence there must be as many even parts as there are odd parts, i.e., $m\geq j$. Let 
$$\pi_{m,j}:=(2m,2m-2,2m-4,\cdots,2j,2j-1,2j-2,2j-3,\cdots,2,1).$$
Clearly $\ell(\pi_{m,j})=m+j$ is even, so all its odd parts are even-indexed, hence it is a particular partition generated by $\tilde{P}_{2N+\nu}(0,j,m,q)$, and $|\pi_{m,j}|=j^2+m(m+1)$. 

Now for any partition $\pi$ generated by $\tilde{P}_{2N+\nu}(0,j,m,q)$, we decompose it uniquely as $\pi=\pi^1+\pi^2$ via the dual map of $\psi_k$ as in Theorem~\ref{decomp1}. More precisely, whenever the gap between two consecutive parts of $\pi$ is $g>2$, we remove $2\lceil\frac{g-2}{2}\rceil$ columns to reduce the gap down to being $1$ or $2$. The removed columns will assemble a partition into at most $m+j$ parts, all being even, and each $\leq 2N+\nu-2m$, we denote it as $\pi^2$, and clearly $\pi^2$ is generated by ${N+j\brack j+m}_{q^2}$. The remaining partition we denote as $\pi^1$, which is still a partition generated by $\tilde{P}_{2N+\nu}(0,j,m,q)$ since this decomposition preserves the parity of every part of $\pi$, and note that all the gaps between consecutive parts of $\pi^1$ are either $1$ or $2$. See Figure~\ref{fig:indent} for illustration, where the removed columns have been highlighted by an arrow $\downarrow$.

Next, $\pi^1$ can be uniquely decomposed into $\pi_{m,j}$ and a partition $\tilde{\pi}$ with parts all divisible by $4$, and each $\leq (m-j)/2$, $\ell(\tilde{\pi})\leq j$. Clearly $\tilde{\pi}$ is generated by ${\lfloor \frac{m+j}{2}\rfloor\brack j}_{q^4}$. A good way to understand this decomposition is to view $\pi^1$ as being built up from $\pi_{m,j}$, by ``moving up'' odd parts in $\pi_{m,j}$, and note that each time an odd part have to ``jump over'' an even number (say $2s$) of even parts, so that it is still even-indexed. Then this ``jump'' is recorded as $q^{4s}$ and contributes to $\tilde{\pi}$. See Figure~\ref{fig:jump} for illustration, where the odd parts in both $\pi^1$ and $\pi_{m,j}$ have been highlighted by an arrow $\downarrow$.

Putting together $\pi_{m,j}$, $\tilde{\pi}$ and $\pi^2$ completes the proof for $m+j$ even. And the case of $m+j$ odd can be derived similarly by noting that $\pi_{m,j}$ should now be replaced by
$$
\pi_{m,j}^{*}:=(2m,2m-2,2m-4,\cdots,2j+2,2j+1,2j,2j-1,\cdots,3,2).
$$
The proof of the other two formulae can be given similarly and thus omitted.
\end{proof}

\begin{figure}
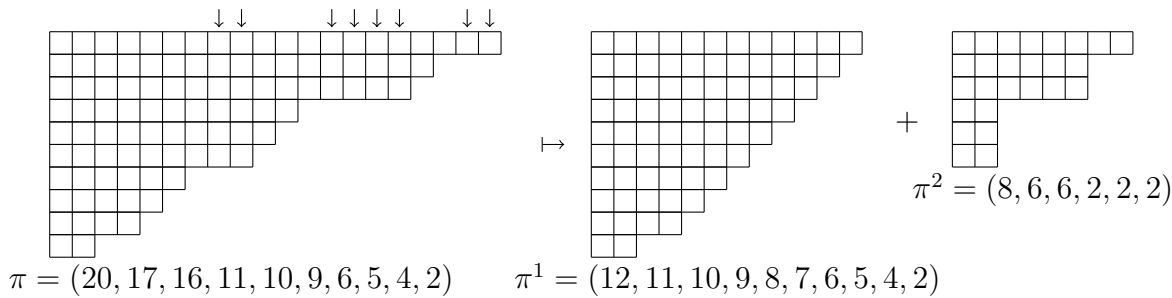

\begin{ferrers}[0.6]
 \addcellrows{20+17+16+11+10+9+6+5+4+2}   
 \highlightcolumnbyarrow{8}
 \highlightcolumnbyarrow{9}
 \highlightcolumnbyarrow{13}
 \highlightcolumnbyarrow{14}
 \highlightcolumnbyarrow{15}
 \highlightcolumnbyarrow{16}
 \highlightcolumnbyarrow{19}
 \highlightcolumnbyarrow{20}
 \transformto
 \putright
 \addcellrows{12+11+10+9+8+7+6+5+4+2}
 \putright
 \addcellrows{8+6+6+2+2+2}
 \addtext{-1}{-2}{$+$}
 \addtext{-16}{-5.5}{$\pi=(20,17,16,11,10,9,6,5,4,2)$}
 \addtext{-5}{-5.5}{$\pi^1=(12,11,10,9,8,7,6,5,4,2)$}
 \addtext{2}{-3.5}{$\pi^2=(8,6,6,2,2,2)$}
 
\end{ferrers}
\caption{$\pi=\pi^1+\pi^2$}
\label{fig:indent}
\end{figure}

\begin{figure}
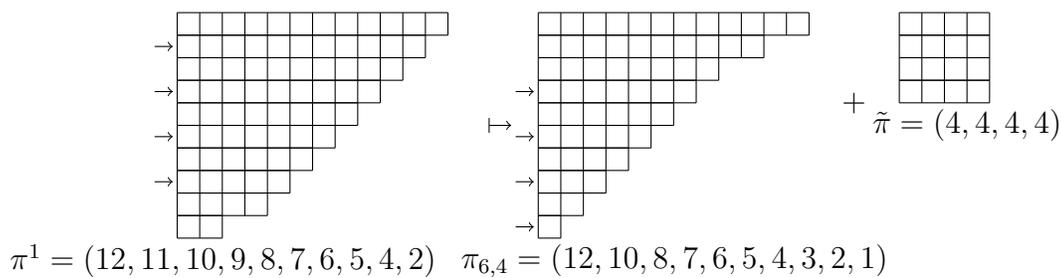

\begin{ferrers}[0.6]
 \addcellrows{12+11+10+9+8+7+6+5+4+2}
 \addtext{1}{-5.5}{$\pi^1=(12,11,10,9,8,7,6,5,4,2)$}
 \highlightrowbyarrow{2}
 \highlightrowbyarrow{4}
 \highlightrowbyarrow{6}
 \highlightrowbyarrow{8}
 \transformto
 \putright
 \addcellrows{12+10+8+7+6+5+4+3+2+1}
 \addtext{3}{-5.5}{$\pi_{6,4}=(12,10,8,7,6,5,4,3,2,1)$}
 \highlightrowbyarrow{4}
 \highlightrowbyarrow{6}
 \highlightrowbyarrow{8}
 \highlightrowbyarrow{10}
 \putright
 \addcellrows{4+4+4+4}
 \addtext{1.5}{-2.5}{$\tilde{\pi}=(4,4,4,4)$}
 \addtext{-1}{-2}{$+$}
\end{ferrers}
\caption{$\pi^1=\pi_{6,4}+\tilde{\pi}$}\label{fig:jump}
\end{figure}

\section{Final remarks }{\label{sec: modk}}

When $k=3$, Theorem~\ref{bu-k} reduces to Berkovich-Uncu's companion to 
Capparelli's identities.
So, we may  ask the  reverse question:  what are the 
 Capparelli type companions to Theorem~\ref{bu-k}? And do they possess Lie theoretical implication as the original Capparelli's identities?

Berkovich and Uncu~\cite{bu1,bu2} derived their results by first stating the explicit 
enumerative formulae of one side of their equations 
and then  prove by induction  that both sides satisfy the same recurrence 
relation.    
This  is reminiscent to the situation in \cite{iz}, where  
the difficult part is to find an explicit solution to  a finite difference equation, once 
a solution is found the proof is routine by checking the recurrence. 
In this paper  we provide a unified combinatorial  approach to
 the generating function versions  of the results in \cite{bu1,bu2} as well as 
the  bounded versions \eqref{iz1} and \eqref{iz2} of Boulet's formulae.

Lastly, for a partition theorem as Theorem~\ref{bu-k}, one naturally has a craving for purely bijective proof. We remark that our proof of \eqref{gfSk} is indeed bijective. But when we derive \eqref{gfDSk} from \eqref{gfSk}, the simple algebraic operation of cancelling the common factor $1/(w_k;w_k)_{\infty}$ will inevitably obscure the bijection. This leaves the problem of finding purely bijective proof of Theorem~\ref{bu-k} still open.

\section*{Acknowledgement} 
This work 
was supported by the LABEX MILYON (ANR-10-LABX-0070) of Universit\'e de Lyon, within the program ``Investissements d'Avenir'' (ANR-11-IDEX-0007) operated by the French National Research Agency (ANR),  and was done during the first author's visit to Universit\'e Claude Bernard Lyon 1, he would like to thank the second author and the Institut Camille Jordan for the hospitality extended during his stay. 
 
  The authors would like to thank George~E.~Andrews for his helpful comments on the manuscript. We are grateful to 
   three anonymous referees for their thorough and constructive reports that improve our manuscript to a great extent.

The first author's research was supported by the Milyon project, the Fundamental Research Funds for the Central Universities (No.~CQDXWL-2014-Z004) and the National Science Foundation of China (No.~11501061).

\small

\end{document}